\documentclass{article}
\usepackage{amsmath}
\usepackage{amsfonts}
\usepackage{amssymb}
\usepackage{amsthm}
\usepackage{graphicx}
\usepackage{hyperref}
\usepackage{enumerate}

\usepackage{xcolor}

\newtheorem{Thm}{Theorem}[section]

\newtheorem{Cor}[Thm]{Corollary}

\newtheorem{Prop}[Thm]{Proposition}

\theoremstyle{definition}
\newtheorem{Def}[Thm]{Definition}
\newtheorem{Prob}[Thm]{Problem}
\theoremstyle{remark}
\newtheorem{Ex}[Thm]{Example}
\newtheorem{Rmk}[Thm]{Remark}
\numberwithin{equation}{section}

\usepackage[T1]{fontenc}
\usepackage[utf8]{inputenc}

\begin{document}
	
	\title{Recovery of the matrix potential of the one-dimensional Dirac equation from spectral data}
	\author{E. Roque, S.M. Torba \\{\small Departamento de Matem\'{a}ticas, Cinvestav, Unidad Quer\'{e}taro, }\\{\small Libramiento Norponiente \#2000, Fracc. Real de Juriquilla,
			Quer\'{e}taro, Qro., 76230 MEXICO.}\\{\small e-mail: earoque@math.cinvestav.edu.mx, storba@math.cinvestav.edu.mx}}
	
	\maketitle
	
	\begin{abstract}
		A method for solving an inverse spectral problem for the one-dimensional Dirac equation is developed. The method is based on the Gelfand-Levitan equation and the Fourier-Legendre series expansion of the transmutation kernel. A linear algebraic system of equations is obtained, which can be solved numerically. To the best of our knowledge, this is the first practical method for the solution of the inverse problem for the one-dimensional Dirac equation on a finite interval.
	\end{abstract}
	
	\section{Introduction}
	
	We consider the one-dimensional Dirac equation in canonical form \cite{diracbook}
	\begin{equation}\label{eqn:dirac}
		\left[\tilde{\sigma}_1 \frac{1}{i}\frac{d}{dx}+\tilde{\sigma}_2p(x)+\tilde{\sigma}_3q(x)\right]Y \equiv \left[ B\frac{d}{dx}+Q(x)\right]Y=\lambda Y, \quad Y=\begin{pmatrix}
			y_1\\
			y_2
		\end{pmatrix}, \quad 0<x<1,
	\end{equation}
	where
	\begin{equation}
		\tilde{\sigma}_1 = \begin{pmatrix}
			0 & i \\ -i & 0
		\end{pmatrix}, \quad \tilde{\sigma}_2 = \begin{pmatrix}
			1 & 0 \\ 0 & -1
		\end{pmatrix}, \quad \tilde{\sigma}_3=\begin{pmatrix}
			0 & 1 \\ 1 & 0
		\end{pmatrix},
	\end{equation}
	are the Pauli matrices (with the naming convention appearing in \cite{diracbook}), and
	\begin{equation}
		B=\begin{pmatrix}
			0 & 1 \\ -1 & 0
		\end{pmatrix}, \quad
		Q(x)=\begin{pmatrix}
			p(x) & q(x) \\
			q(x) & -p(x)
		\end{pmatrix}, \quad p,q\in L^2([0,1],\mathbb{R}).
	\end{equation}
	We write the Pauli matrices with tilde to distinguish them from the more standard notation \(\sigma_1, \sigma_2, \sigma_3,\) used in other works such as \cite{Grant, Hryniv-radial}, where \( \sigma_1 = \tilde{\sigma}_3, \sigma_2= -\tilde{\sigma}_1, \sigma_3=\tilde{\sigma}_2. \)\\
	
	The inverse problem for the one-dimensional Dirac equation \eqref{eqn:dirac} consists of recovering the matrix potential \(Q\) from the spectral data of a suitable boundary value problem; see the precise statements of problem \ref{prob:bothbc} and problem \ref{prob:onebc}. A multitude of methods to numerically solve the inverse problem for the Schrödinger equation are known; see \cite{RundellSacks,Marletta,IgnatievYurko,Gao,Drignei,kravchenko-gb,KravTorbaInv}  and the references therein. However, to the best of our knowledge, there are no practical methods available for the case of the Dirac equation. The aim of this work is to develop a method of solution for the inverse problems \ref{prob:bothbc} and \ref{prob:onebc} that is suitable for numerical realization. \\
	
	Ever since the seminal work of Gasymov and Levitan \cite{Gasymov1, Gasymov2, Gasymov3}, spectral problems for Dirac-type operators have received significant attention and several works have been written, see \cite{marchenko, levitan, diracbook, Hryniv, HrynivManko, Shkalikov, Lunyov} and the references therein. Moreover, due to its connection with the ZS-AKNS system and its applications to the inverse scattering transform \cite{ZS1,ZS2,HrynivManko,Lunyov, AKNSMatrix}, spectral and nodal problems for Dirac operators are very active research topics up to this day \cite{Makin, WeiNodal, Wei}. \\
	
	The method of solution of the inverse problems \ref{prob:bothbc} and \ref{prob:onebc} proposed in the present work is based on the Gelfand-Levitan equation, which relates a matrix-valued function \(F\) containing the spectral data of the problem with the integral kernel \(K\) of the transmutation operator. Inspired by the method presented in \cite{Krav2019, KravTorbaInv} and the recently obtained Fourier-Legendre series expansion for the transmutation kernel \cite{RoqTorNSBF}, we obtain a system of linear algebraic equations for the Fourier-Legendre coefficients of \(K\).  We then show that the matrix potential \(Q\) can be recovered from the first coefficient. \\
	
	We illustrate that the method developed here is accurate and efficient in various examples, including smooth and non-smooth potentials. The practical solution of the inverse problem opens up the possibility for the recovery of radial Dirac operators. Indeed, in the reconstruction algorithm presented in \cite[p. 9]{Hryniv-radial}, an inverse spectral problem for the radial Dirac equation is reduced to an inverse problem for a regular Dirac operator, see also \cite{Serier}. The present paper is the first step for extending several methods available for inverse problems for the Schrödinger equation \cite{kravchenko-gb, krav-icp} to the Dirac equation  as well as for the realization of the inverse scattering transform method for the ZS-AKNS system. \\
	
	The paper is organized as follows. In section \ref{sec:prelim} we recall several results regarding the integral representation of solutions in terms of transmutation operators. Furthermore, we present the characterization of the spectral data, and we state the inverse problems considered in this work. In section \ref{sec:nsbf} we expand each of the kernels \(K_C(x,t)\), \(K_S(x,t)\) and \(K_\alpha(x,t)\) of the corresponding transmutation operators in a matrix Fourier-Legendre series. Additionally, we derive the Neumann series of Bessel functions (NSBF) representation for different solutions. In section \ref{sec:GL}, we state the Gelfand-Levitan equation for the kernels \(K_C(x,t)\), \(K_S(x,t)\), and \(K_\alpha(x,t)\). The Fourier-Legendre matrix expansions of the kernels are used to obtain the main result of this work: a linear system of algebraic equations for their respective coefficients. We prove that the matrix potential \(Q\) can be recovered from the first coefficient of the matrix Fourier-Legendre series of the transmutation kernels. Finally, numerical examples are presented in section \ref{sec:numerical}. Appendix \ref{section:appendix-bg} summarizes known results regarding the Bubnov-Galerkin method and its stability for the overall readability of the paper.
	
	\section{Preliminaries}\label{sec:prelim}
	Throughout this work, we denote by \( C(\lambda, x)=(C_1(\lambda, x), C_2(\lambda, x))^T \) the solution of \eqref{eqn:dirac} with initial condition \( C(\lambda, 0) = (1,0)^T \). Additionally, we denote by \( S(\lambda, x)=(S_1(\lambda, x), S_2(\lambda, x))^T \) the solution of \eqref{eqn:dirac} with initial condition \( S(\lambda, 0) = (0,1)^T \). In the case we have the null potential \(Q \equiv 0\), we have a pair of fundamental solutions given by \( C_0(\lambda, x):=( \cos \lambda x, \sin \lambda x)^T \) and \( S_0(\lambda, x)=(-\sin \lambda x, \cos \lambda x)^T \) that satisfy the same initial conditions as \(C(\lambda,x)\) and \(S(\lambda,x)\) respectively.
	
	\paragraph*{Notations}
	By \( \mathcal{M}_2 \) we denote the algebra of \( 2\times 2 \) matrices with complex valued entries endowed with the operator norm \(\vert \cdot \vert \) induced by the Euclidean norm of the space \( \mathbb{C}^2 \), see \cite{matrix} for details regarding matrix norms and their properties. For \(A \in \mathcal{M}_2\), we write \(A^T\) and \(A^*\) for the transpose and the conjugate transpose of \(A\) respectively. Moreover, \( L^2([a,b],\mathcal{M}_2) \) denotes the space of \( \mathcal{M}_2 \)-valued functions with finite norm
	\begin{equation}\label{eqn:norm2}
		\Vert A \Vert_{L^2([a,b])}:= \left( \int_{a}^{b} \left\vert A(t) \right\vert^2 dt \right)^{1/2}.
	\end{equation}
	\begin{Rmk}
		It can be verified that the matrix operator norm induced by the Euclidean norm of the space \( \mathbb{C}^2 \) satisfies
		\begin{equation}
			\Big\vert \int_a^b A(t)\; dt \Big\vert \leq \int_{a}^b \vert A(t) \vert dt.
		\end{equation}
	\end{Rmk}
	Additionally, \( W^{p,2}([a,b],\mathcal{M}_2)\) denotes the space of all matrix-valued functions \(A\) whose derivatives of order at most \(p-1\) are absolutely continuous functions and \(A^{(p)}\in L^2([a,b],\mathcal{M}_2) \). Also, we set
	\begin{equation}
		\mathcal{A}_Q:=B \frac{d}{dx} +Q(x), \quad \mathcal{A}_0:=B \frac{d}{dx} .
	\end{equation}
	
	The reader is invited to consult the monographs \cite{levitan, marchenko, diracbook} for basic definitions and results about transmutation operators and spectral theory of differential operators. The existence of transmutation operators for the Dirac operator and different classes of potentials has been found in several publications, for instance see \cite{levitan, diracbook, Hryniv, Nelson-analytic}. The following theorem is a particular case of the results presented in \cite{diracTOP}, see also \cite{diracbook, Nelson-analytic, RoqTorNSBF}.
	\begin{Thm}[\cite{diracTOP,diracbook}] \label{thm:TOP}
		Let  \( Q\in L^2([0,1], \mathcal{M}_2) \) and let \(I\) denote the \(2\times 2\) identity matrix. Then there exists a matrix-valued function \(K(x,t)\) such that the matrix solution \( U(\lambda,x)\) of the Cauchy problem
		\begin{equation}\label{eqn:U}
			\mathcal{A}_Q U = \lambda U, \quad U(\lambda,0)=I,
		\end{equation}
		can be represented in the form
		\begin{equation}
			U(\lambda,x) = e^{-B\lambda x } + \int_{-x}^x K(x,t) e^{-B\lambda t} dt,
		\end{equation}
		and the relations
		\begin{align}
			K(x,x)B-BK(x,x)&=Q(x), \\
			K(x,-x)B+BK(x,-x)&=0,
		\end{align}
		hold for almost all \(x \in [0,1].\) Moreover, \(K(x,\cdot) \in L^2([-x,x],\mathcal{M}_2)\).
	\end{Thm}
	\begin{Rmk}\label{rmk:U0}
		Notice that
		\begin{equation}
			U_0(\lambda,x):= e^{-B\lambda x } = \begin{pmatrix}
				\cos \lambda x & -\sin \lambda x \\
				\sin \lambda x & \cos \lambda x
			\end{pmatrix},
		\end{equation}
		is the matrix solution of the initial value problem \( \mathcal{A}_0 U_0= \lambda U_0, \; U_0(\lambda,0)=I\). Moreover, the first and second column of the matrix \(U_0(\lambda, x)\) coincide with \(C_0(\lambda,x)\) and \(S_0(\lambda, x)\) respectively.
	\end{Rmk}
	Transmutation operators with upper triangular kernels play a key role in the solution of inverse spectral problems.
	Denote by  \( \varphi_\alpha(\lambda,x)\) the solution to \eqref{eqn:dirac} that satisfies the initial condition
	\begin{equation}
		\varphi_\alpha(\lambda,0) = \begin{pmatrix}
			\sin \alpha \\
			-\cos \alpha
		\end{pmatrix}.
	\end{equation}
	Then, we have the following assertion regarding the existence of a transmutation operator with upper-triangular kernel which gives an integral representation of the solution \( \varphi_\alpha \).
	\begin{Thm}[{\cite[Thm 2.2, Thm 2.3]{diracbook}}]\label{thm:Kalpha}
		Let us assume that \(Q \in L^2([0,1], \mathcal{M}_2)\). Then,
		the solution \(\varphi_\alpha(\lambda,x)\) can be represented in the form
		\begin{equation}\label{eqn:TOPKalpha}
			\varphi_\alpha(\lambda,x) = \varphi_{\alpha,0}(\lambda,x)  + \int_{0}^x K_\alpha(x,t) \varphi_{\alpha,0}(\lambda,t) dt,
		\end{equation}
		where
		\begin{equation}
			\varphi_{\alpha,0}(\lambda,x) = \begin{pmatrix}
				\sin (\lambda x +\alpha) \\
				-\cos ( \lambda x + \alpha)
			\end{pmatrix},
		\end{equation}
		and \( K_\alpha(x,\cdot) \in L^2(0,x).\) Moreover,
		\begin{equation}\label{eqn:Kalpha}
			K_\alpha(x,t)=K(x,t)-K(x,-t)\left(\tilde{\sigma}_2 \cos 2\alpha +\tilde{\sigma}_3 \sin 2\alpha \right),
		\end{equation}
		and for almost all \(x \in [0,a]\)
		\begin{equation}\label{eqn:KAalphxx}
			K_{\alpha,A}(x,x)B-BK_{\alpha,A}(x,x)=Q(x),
		\end{equation}
		where \(K_{\alpha,A}\) denotes the part of \(K_\alpha\) that anticommutes with \(B\). In case of a continuous potential, this is equivalent to
		\begin{equation}\label{eqn:Kalphxx}
			K_{\alpha}(x,x)B-BK_{\alpha}(x,x)=Q(x).
		\end{equation}
	\end{Thm}
	\begin{Rmk}
		Recall that given a \(2\times 2\) matrix-valued function \(A\), we can consider the operators \(\mathcal{P}^+\) and \(\mathcal{P}^-\) defined as \cite{Nelson-analytic}:
		\begin{equation}
			\mathcal{P}^-[A]:=\frac{1}{2}(A-BAB), \quad \mathcal{P}^+[A]:=\frac{1}{2}(A+BAB).
		\end{equation}
		Then, it follows that \(\mathcal{P}^-\) and \(\mathcal{P}^+\) are projectors, and
		\[
		B\mathcal{P}^-[A] = \mathcal{P}^-[A]B, \quad B\mathcal{P}^+[A]=-\mathcal{P}^+[A]B.
		\]
		Hence, \(K_{\alpha,A}(x,x)=\mathcal{P}^+[K_\alpha(x,x)]\).
	\end{Rmk}
	There are two values of \(\alpha\) which are of particular importance, \(\alpha = -\pi/2 \) and \(\alpha =0.\)   An immediate consequence of theorem \ref{thm:Kalpha} is the following.
	\begin{Cor}
		The pair of fundamental solutions \( C(\lambda,x)\) and \(S(\lambda,x)\) can be represented in the form
		\begin{align}
			C(\lambda,x) &=C_0(\lambda, x) + \int_{0}^{x}K_C(x,t)C_0(\lambda, t)dt, \label{eqn:Kc} \\
			S(\lambda,x) &= S_0(\lambda,x)+\int_{0}^{x} K_S(x,t)S_0(\lambda,t)dt \label{eqn:Ks},
		\end{align}
		where
		\begin{align}
			K_C(x,t)&=K(x,t)+K(x,-t)\tilde{\sigma}_2,  \\
			K_S(x,t)&=K(x,t)-K(x,-t)\tilde{\sigma}_2.
		\end{align}
		The kernel \(K_C\) corresponds to the case \(\alpha = -\frac{\pi}{2}\) in \eqref{eqn:Kalpha}, whereas \(K_S\) corresponds to the case \(\alpha=0\).
	\end{Cor}
	Now, let us consider the following boundary value problem, which we will denote \( L(p,q,\alpha,\beta)\).
	\begin{align}
		&\mathcal{A}_Q Y = \lambda Y,  \\
		&y_1(0) \cos \alpha+  y_2(0) \sin \alpha=0, \label{eqn:bc0} \\
		&y_1(1) \cos \beta +  y_2(1) \sin \beta=0. \label{eqn:bc1}
	\end{align}
	The next results summarize the characterization of the spectral data for the problem \( L(p,q,\alpha,\beta)\).
	\begin{Thm}[\cite{diracbook, levitan, Hryniv}]\label{thm:asym-lambda}
		Let \( Q \in L^2([0,1], \mathcal{M}_2)\) be a real matrix-valued potential, and let \( \alpha, \beta \in \mathbb{R} \). Then, the problem \( L(p,q,\alpha,\beta)\) has a countable set of eigenvalues \( \lambda_m = \lambda_{m}(p,q,\alpha,\beta), \, m \in \mathbb{Z}\), which forms an unbounded sequence
		\begin{equation}\label{eqn:lasymL2}
			\lambda_{m} = m\pi + (\beta - \alpha) + h_m, \quad \{ h_m \}_{m \in \mathbb{Z}}\in l^2.
		\end{equation}
		Moreover, if \( Q \in W^{k,2}([0,1], \mathcal{M}_2)\), the following asymptotics hold
		\begin{equation}\label{eqn:asym-eig}
			\lambda_{m} = m \pi +(\beta - \alpha) + \frac{c_1}{m} + \frac{c_2}{m^2} + \ldots + \frac{c_k}{m^k} + o\left(\frac{1}{m^k}\right), \quad \vert m \vert \to \infty
		\end{equation}
		where \( c_1, c_2, \ldots, c_k\) are some constants.
	\end{Thm}
	Due to the periodic nature of trigonometric functions, adding an integer multiple of \(\pi\) to alpha or beta does not change the boundary conditions. To fix an enumeration of the spectrum we consider \(-\pi/2 \leq \alpha, \beta < \pi/2,\) so that in the case of the null potential \(Q \equiv 0\) we get \(\lambda_0 = \beta - \alpha \in (-\pi,\pi)\). Another reason for the election of this interval for \(\alpha\) is to have consistency between the eigenvalue asymptotics appearing in \cite{Hryniv} and \cite{diracbook}. On the other hand, since the solution \(\varphi_\alpha(\lambda,x)\) satisfies the boundary condition \eqref{eqn:bc0}, it follows that \( \varphi_\alpha(\lambda_m,x)\) are the eigenfunctions of the problem \(L(p,q,\alpha,\beta)\). The norming constants are defined as the squares of the \(L^2\)- norms of the eigenfunctions, that is
	\begin{equation}\label{eqn:norming}
		\alpha_m:= \int_{0}^{1} \vert \varphi_\alpha(\lambda_m,x) \vert^2 dx.
	\end{equation}
	The norming constants possess the following asymptotic formulas.
	\begin{Thm}[\cite{diracbook, Hryniv}]\label{thm:asym-nc}
		Let \(Q \in L^2([0,1],\mathcal{M}_2)\) be a real matrix-valued potential. Then,
		\begin{equation}\label{eqn:asym-normingcr}
			\alpha_m = 1 + r_m, \quad \{r_m\}_{m \in \mathbb{Z}} \in l^2.
		\end{equation}
		For a potential \( Q \in W^{k,2}([0,1], \mathcal{M}_2) \), the following asymptotics are true
		\begin{equation}\label{eqn:asym-normingc}
			\alpha_m = 1 + \frac{a_1}{m}+ \frac{a_2}{m^2} + \ldots + \frac{a_k}{m^k} + o\left(\frac{1}{m^k}\right), \quad \vert m \vert \to \infty,
		\end{equation}
		where \(a_1,a_2, \ldots, a_k\) are some constants.
	\end{Thm}
	\begin{Rmk}\label{rmk:beta-rec}
		Let us fix one of the boundary conditions, for instance, let us fix the value of \(\alpha\). Given a sequence of eigenvalues \( \lambda_{m} = \lambda_{m}(p,q,\alpha, \beta)\), as we can see from the asymptotics of the spectrum \eqref{eqn:lasymL2},
		\begin{equation}
			\lim_{\vert m \vert \to \infty } \left( \lambda_{m} - m \pi  \right)	 =  (\beta - \alpha).
		\end{equation}
		Thus, the second boundary condition is uniquely determined by the asymptotic formula for the spectrum and the first boundary condition. Hence, if \(\lambda_{m}(p_1,q_1,\alpha,\beta_1) = \lambda_{m}(p_2,q_2,\alpha,\beta_2) \) for all \( m \in \mathbb{Z}\) then \( \beta_1 = \beta_2.\) Similarly, if we fix the value of \(\beta\), then \( \alpha\) is uniquely determined by the asymptotics of the spectrum.
	\end{Rmk}
	Let us formulate the inverse Dirac problems considered in this work.
	\begin{Prob}[Recovery of a Dirac problem from its spectral data] \label{prob:bothbc}
		Given two sequences of real numbers \( \{\lambda_{m}\}_{m=-\infty}^\infty\) and \( \{\alpha_{m}\}_{m=-\infty}^\infty\) that satisfy the asymptotics \eqref{eqn:lasymL2} and \eqref{eqn:asym-normingcr} respectively, and given the parameters \(\alpha\) and \(\beta\) appearing in the boundary conditions \eqref{eqn:bc0} and \eqref{eqn:bc1}, find a real matrix-valued function \(Q\) such that \( \{\lambda_{m}\}_{m=-\infty}^\infty\) and \( \{\alpha_{m}\}_{m=-\infty}^\infty\) are the sequences of eigenvalues and norming constants of the problem \(L(p,q,\alpha,\beta).\)
	\end{Prob}
	Due to remark \ref{rmk:beta-rec}, we will also consider the following inverse spectral problem.
	\begin{Prob}[Recovery of a Dirac problem from its spectral data and one boundary condition]\label{prob:onebc}
		Given two sequences of real numbers \( \{\lambda_{m}\}_{m=-\infty}^\infty\) and \( \{\alpha_{m}\}_{m=-\infty}^\infty\) that satisfy the asymptotics \eqref{eqn:lasymL2} and \eqref{eqn:asym-normingcr} respectively, and given the parameter \(\alpha\) appearing in the boundary conditions \eqref{eqn:bc0}, find a real matrix-valued function \(Q\) and a real number \(\beta\) such that \( \{\lambda_{m}\}_{m=-\infty}^\infty\) and \( \{\alpha_{m}\}_{m=-\infty}^\infty\) are the sequences of eigenvalues and norming constants of the problem \(L(p,q,\alpha,\beta).\)
	\end{Prob}
	For  results regarding one-dimensional Dirac operators, the reader is invited to consult, for instance, \cite{levitan,diracbook, Hryniv} and the references therein. In particular, the reader is encouraged to review chapters 6 and 7 of the book \cite{diracbook}, regarding inverse problems, uniqueness theorems and isospectral Dirac operators.
	
	\section{Fourier-Legendre series for the integral kernels}\label{sec:nsbf}
	Denote by \(P_n\) the \(n\)-th Legendre polynomial. For a potential \(Q \in L^2([0,1], \mathcal{M}_2)\), since \( K(x,\cdot) \in L^2((-x,x),\mathcal{M}_2)\) it admits a Fourier-Legendre series expansion \cite{RoqTorNSBF}:
	\begin{equation}\label{eqn:K}
		K(x,t)= \sum_{n=0}^\infty \frac{1}{x} K_n(x)P_n\left(\frac{t}{x}\right), \quad K_n(x):=\begin{pmatrix}
			K^{(n)}_{11}(x) & K^{(n)}_{12}(x) \\ K^{(n)}_{21}(x) & K^{(n)}_{22}(x)
		\end{pmatrix},
	\end{equation}
	and the series converges with respect to the norm of the space \(L^2((-x,x),\mathcal{M}_2)\) for every \(x>0\). Using \eqref{eqn:Kalpha} and the fact that \(P_n(-t)=(-1)^{n}P_n(t)\)  we can obtain a Fourier-Legendre series for \(K_\alpha\) as follows
	\begin{equation}\label{eqn:KalphaFLS1}
		\begin{split}
			K_\alpha(x,t)&=\sum_{n=0} \frac{1}{x}K_n(x)\left[I+(-1)^{n+1}C_\alpha \right]P_n\left(\frac{t}{x}\right)\\
			&=\sum_{n=0} \frac{1}{x}K_{2n}(x)\left[I-C_\alpha \right]P_{2n}\left(\frac{t}{x}\right)+\sum_{n=0} \frac{1}{x}K_{2n+1}(x)\left[I+C_\alpha \right]P_{2n+1}\left(\frac{t}{x}\right),
		\end{split}
	\end{equation}
	where \(C_\alpha = \tilde{\sigma}_2 \cos 2\alpha + \tilde{\sigma}_3 \sin 2\alpha.\)
	In particular,
	\begin{align*}
		K_C(x,t)&=\sum_{n=0}^{\infty} \frac{1}{x} \left[ K_{2n}(x)[I+\tilde{\sigma}_2]P_{2n}\left( \frac{t}{x} \right) + K_{2n+1}(x)[I-\tilde{\sigma}_2]P_{2n+1}\left( \frac{t}{x} \right) \right] \\
		&=2 \sum_{n=0}^{\infty} \frac{1}{x} \begin{pmatrix} K_{11}^{(2n)}(x)P_{2n}\left( \frac{t}{x} \right) & K_{12}^{(2n+1)}(x)P_{2n+1}\left( \frac{t}{x} \right) \\ K_{21}^{(2n)}(x)P_{2n}\left( \frac{t}{x} \right) & K_{22}^{(2n+1)}(x)P_{2n+1}\left( \frac{t}{x} \right) \end{pmatrix} \\
		& = 2 \sum_{n=0}^{\infty} \frac{1}{x} K^C_n(x)A_n^C(x,t),
	\end{align*}
	where
	\begin{equation}\label{eqn:Kcn}
		K^C_n(x):=\begin{pmatrix} K_{11}^{(2n)}(x) & K_{12}^{(2n+1)}(x) \\ K_{21}^{(2n)}(x) & K_{22}^{(2n+1)}(x) \end{pmatrix}, \quad 	 A_n^C(x,t):=\begin{pmatrix} P_{2n}\left( \frac{t}{x} \right) & 0 \\ 0 & P_{2n+1}\left(\frac{t}{x} \right) \end{pmatrix}.
	\end{equation}
	Similarly, we can show that the kernel \(K_S(x,\cdot)\) admits a series representation given by
	\[
	K_S(x,t) = 2 \sum_{n=0}^{\infty} \frac{1}{x} K_n^S(x)A_n^S(x,t),
	\]
	where
	\begin{equation}\label{eqn:Ksn}
		K^S_n(x):=\begin{pmatrix} K_{11}^{(2n+1)}(x) & K_{12}^{(2n)}(x) \\ K_{21}^{(2n+1)}(x) & K_{22}^{(2n)}(x) \end{pmatrix}, \quad 	 A_n^S(x,t):=\begin{pmatrix} P_{2n+1}\left( \frac{t}{x} \right) & 0 \\ 0 & P_{2n}\left(\frac{t}{x} \right) \end{pmatrix}.
	\end{equation}
	We summarize these results in the following proposition.
	\begin{Prop}\label{prop:KCKS}
		The integral kernels \(K_C(x,\cdot)\) and \(K_S(x,\cdot)\) admit matrix Fourier-Legendre series representations
		\begin{equation}\label{eqn:KCKS}
			K_C(x,t) = 2 \sum_{n=0}^{\infty} \frac{1}{x} K^C_n(x)A_n^C(x,t), \quad K_S(x,t) = 2 \sum_{n=0}^{\infty} \frac{1}{x} K_n^S(x)A_n^S(x,t),
		\end{equation}
		where \(K_n^C, A_n^C\) and \( K_n^S, A_n^S\) are given by eqs. \eqref{eqn:Kcn} and \eqref{eqn:Ksn} respectively. Moreover, the series converges in the norm of \(L^2([0,x],\mathcal{M}_2)\), which follows from the convergence of \eqref{eqn:K}.
	\end{Prop}
	One may ask if it is possible to write the Fourier-Legendre series for \(K_\alpha(x,t)\) in a compact way such as those for \(K_C(x,t)\) and \(K_S(x,t)\). We give a positive answer for this question, but first we need an auxiliary proposition.
	\begin{Prop}\label{prop:RaCa}
		Let \(R_\alpha\) be the rotation matrix by an angle \( \alpha \)
		\begin{equation}\label{eqn:Ralpha}
			R_\alpha = \begin{pmatrix}
				\cos \alpha & - \sin \alpha \\ \sin \alpha & \cos \alpha
			\end{pmatrix}.
		\end{equation}
		Consider the matrix \(C_\alpha = \tilde{\sigma}_2 \cos 2\alpha + \tilde{\sigma}_3 \sin 2\alpha.\) Then,
		\begin{enumerate}
			\item \(C_\alpha \tilde{\sigma}_2 = R_{2\alpha}.\)
			\item \( \tilde{\sigma}_2 R_\alpha \tilde{\sigma}_2 = R_{-\alpha}=R_\alpha^T=R_\alpha^{-1}.\)
			\item \( (I+C_\alpha)(I-C_\alpha) = 0. \)
			\item \((I\pm C_\alpha)R_\alpha=R_\alpha(I \pm \tilde{\sigma}_2).\)
		\end{enumerate}
	\end{Prop}
	Right-multiplication of eq. \eqref{eqn:KalphaFLS1} by \(R_\alpha\) and proposition \ref{prop:RaCa}  gives us
	\begin{equation}
		K_\alpha(x,t)R_\alpha=\sum_{n=0}^\infty \frac{1}{x}K_{2n}(x)R_\alpha\left[I-\tilde{\sigma}_2\right]P_{2n}\left(\frac{t}{x}\right)+\sum_{n=0}^\infty \frac{1}{x}K_{2n+1}(x)R_\alpha\left[I+\tilde{\sigma}_2\right]P_{2n+1}\left(\frac{t}{x}\right),
	\end{equation}
	which simplifies to
	\begin{equation}
		K_\alpha(x,t)R_\alpha=2 \sum_{n=0}^\infty \frac{1}{x}K_n^\alpha(x) A_n^S(x,t),
	\end{equation}
	where
	\begin{equation}
		K_{n}^\alpha(x)=\begin{pmatrix}
			\cos \alpha \, K_{11}^{(2n+1)}(x)+\sin \alpha \, K_{12}^{(2n+1)}(x) & -\sin \alpha \, K_{11}^{(2n)}(x)+\cos \alpha \, K_{12}^{(2n)}(x) \\
			\cos \alpha \, K_{21}^{(2n+1)}(x)+\sin \alpha \, K_{22}^{(2n+1)}(x) & -\sin \alpha \, K_{21}^{(2n)}(x)+\cos \alpha \, K_{22}^{(2n)}(x)
		\end{pmatrix}.
	\end{equation}
	Hence,
	\begin{equation}\label{eqn:KalphaASR}
		K_\alpha(x,t)=2 \sum_{n=0}^\infty \frac{1}{x}K_n^\alpha(x) A_n^S(x,t)R_\alpha^T.
	\end{equation}
	\subsection{Neumann series of Bessel functions representation for solutions of the Dirac equation}
	
	In \cite{RoqTorNSBF} a Neumann series of Bessel functions (NSBF) representation was obtained for the matrix solution of the Cauchy problem \eqref{eqn:U}. With a modification of the procedure from \cite{RoqTorNSBF} we can obtain NSBF series representations for \(C(\lambda,x)\) and \(S(\lambda,x)\) where the coefficients \(K_n^C\) and \(K_n^S\) appear explicitly. Since the procedure is pretty similar, we do not present here all the details. Substitution of the series \eqref{eqn:KCKS} into eqs. \eqref{eqn:Kc} and \eqref{eqn:Ks} leads to
	\begin{equation}
		\begin{split}
			C(\lambda,x)&= C_0(\lambda,x)+2\sum_{n=0}^\infty \frac{1}{x} K_n^C(x) \int_{0}^x A_n^C(x,t)C_0(\lambda,t)\;dt, \\
			S(\lambda,x)&= S_0(\lambda,x)+2\sum_{n=0}^\infty \frac{1}{x} K_n^S(x) \int_{0}^x A_n^S(x,t)S_0(\lambda,t)\; dt.
		\end{split}
	\end{equation}
	The integrals above can be exactly computed in terms of spherical Bessel functions thanks to the following formulas \cite[f. 2.17.7, p. 386]{prudnikov2}
	\begin{equation} \label{eqn:intPncs}
		\frac{1}{x} \int_{0}^{x} P_{2n}\left(\frac{t}{x}\right) \cos \lambda t \; dt = (-1)^n j_{2n}(\lambda x), \quad \frac{1}{x} \int_{0}^{x} P_{2n+1}\left(\frac{t}{x}\right) \sin \lambda t \; dt = (-1)^n j_{2n+1}(\lambda x),
	\end{equation}
	where \(j_\nu(z)\) denotes the spherical Bessel function of order \(\nu\). To obtain an NSBF representation for the solution \(\varphi_\alpha(\lambda,x)\), observe that
	\begin{equation}
		\varphi_{\alpha,0}(\lambda,x)=R_\alpha {\varphi}_{0}(\lambda,x), \quad {\varphi}_{0}(\lambda,x):= \begin{pmatrix}
			\sin (\lambda x) \\
			-\cos ( \lambda x)
		\end{pmatrix}.
	\end{equation}
	Thus, substitution of eq. \eqref{eqn:KalphaASR} into \eqref{eqn:TOPKalpha} leads to
	\begin{equation}
		\varphi_\alpha(\lambda,x)=\varphi_{\alpha,0}(\lambda,x)+2\sum_{n=0}^\infty\frac{1}{x}K_n^\alpha(x)\int_{0}^x A_n^S(x,t) \varphi_{0}(\lambda,t)\; dt.
	\end{equation}
	Hence, we get the following NSBF representation for the solutions \(C(\lambda,x),\) \(S(\lambda,x)\) and \(\varphi_\alpha(\lambda,x)\).
	\begin{Thm} The solutions \(C(\lambda,x)\), \(S(\lambda,x)\) and \(\varphi_\alpha(\lambda,x)\) admit the following NSBF representations
		\begin{equation}\label{eqn:nsbfCS}
			\begin{split}
				C(\lambda,x)&= C_0(\lambda,x)+2\sum_{n=0}^\infty (-1)^n K_n^C(x) \begin{pmatrix}
					j_{2n}(\lambda x) \\
					j_{2n+1} (\lambda x)
				\end{pmatrix}, \\
				S(\lambda,x)&= S_0(\lambda,x)+2\sum_{n=0}^\infty (-1)^n K_n^S(x) \begin{pmatrix}
					-j_{2n+1}(\lambda x) \\
					j_{2n} (\lambda x)
				\end{pmatrix}, \\
				\varphi_\alpha(\lambda,x)&= \varphi_{\alpha,0}(\lambda,x)+2\sum_{n=0}^\infty (-1)^n K_n^\alpha(x) \begin{pmatrix}
					j_{2n+1}(\lambda x) \\
					-j_{2n} (\lambda x)
				\end{pmatrix}.
			\end{split}
		\end{equation}
		For every \( \lambda \in \mathbb{C}\) the series converge pointwise. For every \( x \in [0,1]\) the series converges uniformly on any compact set of the complex plane of the variable \(\lambda\) and the remainders of their partial sums admit estimates independent of \( \operatorname{Re } \lambda.\)
	\end{Thm}
	The reader can consult \cite[Section 4]{nsbf} for details about the convergence of Neumann series of Bessel functions.
	\section{The Gelfand-Levitan equation}\label{sec:GL}
	In this section we introduce the Gelfand-Levitan equation for the integral kernels \(K_C\), \(K_S\) and \(K_\alpha\).
	Let \(  \lambda_m \) and  \( \alpha_m \) be the eigenvalues and the norming constants of the boundary value problem  \(L(p,q,-\pi/2,0),\) \(p,q\in L^2([0,1],\mathbb{R})\). Additionally, let \(\lambda_{m}^0\) and \(\alpha_m^0\) denote the eigenvalues and norming constants of the problem \( L(0,0,-\pi/2,0)\). Then it is straightforward to verify that
	\begin{equation}
		\lambda_{m}^0=\left(m+\frac{1}{2}\right)\pi, \quad \alpha_m^0=1, \quad m \in \mathbb{Z}.
	\end{equation}
	\begin{Thm}[{\cite[Lemma 5.2, Thm. 5.5]{Hryniv}}]\label{thm:glm}
		The kernel \( K_C \) of the transmutation operator \eqref{eqn:Kc} transforming the solution \( C_0(\lambda, x) \) to the solution \( C(\lambda, x) \) satisfies the Gelfand-Levitan equation
		\begin{equation}\label{eqn:GL-KC}
			K_C(x,t)+F_C(x,t)+\int_{0}^{x}K_C(x,s)F_C(s,t)ds=0,
		\end{equation}
		where
		\begin{equation}
			F_C(x,t):=\frac{1}{2}\left[ H\left( \frac{x-t}{2} \right)+ H\left( \frac{x+t}{2} \right)\tilde{\sigma}_2 \right],
		\end{equation}
		\begin{equation}\label{eqn:H}
			H(s):=\mathrm{p.v.} \sum_{m=-\infty}^{\infty}H_m(s)= \mathrm{p.v.} \sum_{m=-\infty}^{\infty}\left( \frac{1}{ \alpha_m } e^{-2\lambda_m s B}-e^{-2\lambda_m^0 sB} \right), \quad H\in L^2([-1,1],\mathcal{M}_2),
		\end{equation}
		the series converges in the \(L^2([-1,1],\mathcal{M}_2)\) norm and \( F_C\in L^2([0,x]\times [0,x],\mathcal{M}_2)\). The Gelfand-Levitan equation is a Fredholm integral equation which has a unique solution. Moreover, for each \(N\), consider the \(N\)-th partial sum
		\[
		H^N(s):=\sum_{m=-N}^N H_m(s)
		\]
		of \(H\). Then, for \(H^N\), the Gelfand-Levitan equation possesses a unique solution \( K^N_C\) which uniquely determines a potential \(Q^N\) via \eqref{eqn:Kalphxx}, and \( Q^N \to Q\) in \(L^2(\mathcal{M}_2)\) when \(N \to \infty\).
	\end{Thm}
	
	\begin{Rmk}
		We point out that our definition of the norming constants \eqref{eqn:norming} follows that of \cite{diracbook, levitan}, which differs from the definition in \cite{Hryniv}. This makes a subtle difference in the expression for \(H\) in this work, in comparison with \cite{Hryniv}.
	\end{Rmk}
	\begin{Rmk}
		Although the fact that \(H \in L^2([-1,1],\mathcal{M}_2)\) was proven in the case \(\beta=0\) in \cite[Lemma 5.2]{Hryniv}, the proof remains valid for \(\beta \neq 0\) due to the asymptotics for the eigenvalues and norming constants, see theorems \ref{thm:asym-lambda}, \ref{thm:asym-nc}. Moreover,
		\begin{align*}
			\frac{1}{ \alpha_m } e^{-2\lambda_m s B}-e^{-2\lambda_m^0 sB}&=\left(\frac{1}{ \alpha_m }-1 \right) e^{-2\lambda_m s B}+e^{-2\lambda_m s B}-e^{-2\lambda_m^0 sB}\\
			&=\left(\frac{1}{ \alpha_m }-1 \right) e^{-2\lambda_m s B}+e^{-2\lambda_m^0 sB}\left(e^{-2(\lambda_m-\lambda_m^0) sB}-I\right).
		\end{align*}
		Then, using the asymptotic formulas \eqref{eqn:lasymL2} and \eqref{eqn:asym-normingcr} we can see that
		\[
			\left\vert \frac{1}{ \alpha_m } e^{-2\lambda_m s B}-e^{-2\lambda_m^0 sB} \right \vert \leq \vert r_m \vert + \vert h_m \vert + o(\vert h_m \vert ).
		\]
		Hence, the series \eqref{eqn:H} converges in \(L^2([-1,1], \mathcal{M}_2)\) without the need to use the principal value.
	\end{Rmk}
	Basic matrix operations show that we can write \(F_C(x,t)\) explicitly as
	\begin{equation}\label{eqn:GL-FC}
		\begin{split}
			F_C(x,t)&=\sum_{m=-\infty}^{\infty} \left[ \frac{1}{\alpha_m} \begin{pmatrix}
				\cos \lambda_m x \cos \lambda_m t & \cos \lambda_m x \sin \lambda_m t \\
				\sin \lambda_m x \cos \lambda_m t & \sin \lambda_m x \sin \lambda_m t
			\end{pmatrix} - \begin{pmatrix}
				\cos \lambda_m^0 x \cos \lambda_m^0 t & \cos \lambda_m^0 x \sin \lambda_m^0 t \\
				\sin \lambda_m^0 x \cos \lambda_m^0 t & \sin \lambda_m^0 x \sin \lambda_m^0 t
			\end{pmatrix} \right]\\
			&=\sum_{m=-\infty}^{\infty} \left[ \frac{1}{\alpha_m} M_C(\lambda_{m},x,t)-M_C(\lambda_m^0,x,t) \right],
		\end{split}
	\end{equation}
	where
	\begin{equation}\label{eqn:Mlambxt}
		M_C(\lambda,x,t):=\begin{pmatrix}
			\cos \lambda x \cos \lambda t & \cos \lambda x \sin \lambda t \\
			\sin \lambda x \cos \lambda t & \sin \lambda x \sin \lambda t
		\end{pmatrix}.
	\end{equation}
	\begin{Rmk}\label{rmk:nonconvergent}
		Even though \( F_C\in L^2([0,x]\times [0,x],\mathcal{M}_2)\), is not necessarily  pointwise convergent. Indeed, consider \(x=t=0.\) Then,
		\begin{equation}
			F_C(0,0)=\sum_{m=-\infty}^{\infty} \left[ \frac{1}{\alpha_m} \begin{pmatrix}
				1 & 0 \\
				0 & 0
			\end{pmatrix} - \begin{pmatrix}
				1 & 0 \\
				0 & 0
			\end{pmatrix} \right].
		\end{equation}
		Even in the case of a smooth potential, due to the asymptotic formula \eqref{eqn:asym-normingc}, we have that
		\begin{equation}\label{eqn:a1/m}
			\left( \frac{1}{\alpha_{m}}-1 \right) \sim -\frac{a_1}{m}, \quad \vert m \vert \to \infty.
		\end{equation}
		Hence, the series for \(F_C(0,0)\) diverges whenever \(a_1\neq 0\). However, as we can see from the asymptotic in \eqref{eqn:a1/m}, the convergence of \eqref{eqn:GL-FC} can be improved by considering the principal value of the series, that is, performing summation from \(m=-M\) to \(M\). Doing so, the main asymptotic term \(-a_1/m\) gets canceled for large values of \(\vert m \vert\).
	\end{Rmk}
	In \cite[Chapter 9]{diracbook}, the Gelfand-Levitan equation for the kernel \(K_S\) is derived.
	\begin{Thm}[\cite{diracbook}]\label{thm:GL-KS}
		Let \(  \lambda_m \) and  \( \alpha_m \) be the eigenvalues and norming constants of the boundary value problem  \(L(p,q,0,0), \, p,q\in L^2([0,1],\mathbb{R}).\) Additionally, let \(\lambda_{m}^0=m\pi\) and \(\alpha_m^0=1\) denote the eigenvalues and norming constants of the problem \( L(0,0,0,0)\). For each fixed \(x\in \left( 0, 1\right]\), the kernel \(K_S(x,t)\) of the transmutation operator \eqref{eqn:Ks} satisfies the Gelfand-Levitan equation
		\begin{equation}\label{eqn:GL-S}
			K_S(x,t)+F_S(x,t)+\int_{0}^{x}K_S(x,s)F_S(s,t)ds=0, \quad 0\leq t \leq x \leq 1,
		\end{equation}
		where
		\begin{equation}\label{eqn:GL-FS}
			\begin{split}
				F_S(x,t)=\sum_{m=-\infty}^{\infty} 	&\left[ \frac{1}{\alpha_m} \begin{pmatrix}
					\sin \lambda_m x \sin \lambda_m t & -\sin \lambda_m x \cos \lambda_m t \\
					-\cos \lambda_m x \sin \lambda_m t & \cos \lambda_m x \cos \lambda_m t
				\end{pmatrix} \right.\\
				&\left. - \begin{pmatrix}
					\sin \lambda_m^0 x \sin \lambda_m^0 t & -\sin \lambda_m^0 x \cos \lambda_m^0 t \\
					-\cos \lambda_m^0 x \sin \lambda_m^0 t & \cos \lambda_m^0 x \cos \lambda_m^0 t
				\end{pmatrix} \right],
			\end{split}
		\end{equation}
		and the equation \eqref{eqn:GL-S} possesses a unique solution.
	\end{Thm}
	\begin{Rmk}\label{rmk:FCFSbeta}
		Both eq. \eqref{eqn:GL-KC} and eq. \eqref{eqn:GL-S} presented in \cite{Hryniv} and \cite{diracbook} respectively were derived for \(\beta=0.\) However the form of the matrices appearing in the expressions for \(F_C\) \eqref{eqn:GL-FC} and \(F_S\) \eqref{eqn:GL-FS} actually depends only on \(\alpha\), see \cite[eq. 9.1.12, p. 130]{diracbook}. Thus, they remain valid for the case \(\beta \neq 0\) replacing the eigenvalues and norming constants of the problem \(L(p,q,\alpha,0)\) for those of the problem \(L(p,q,\alpha,\beta)\), for \(\alpha=-\pi/2, \) and \(\alpha=0\) respectively. At the same time, one needs to change the eigenvalues of the problem \(L(0,0,\alpha,0)\) for those of the problem \( L(0,0,\alpha,\beta).\) Note that the norming constants \(\alpha_m^0\) are the same for \(L(0,0,\alpha,0)\) and \( L(0,0,\alpha,\beta).\)
	\end{Rmk}
	Although the Gelfand-Levitan equation was derived for the kernel \(K_S(x,t)\) in \cite{diracbook}, the procedure remains valid for the case of \(K_\alpha(x,t).\) We state the equation here without proof.
	\begin{Thm}\label{thm:GL-Kalph}
		Let \(  \lambda_m \) and  \( \alpha_m \) be the eigenvalues and norming constants of the boundary value problem  \(L(p,q,\alpha,\beta), \, p,q\in L^2([0,1],\mathbb{R}), \; -\pi/2\leq\alpha,\beta<\pi/2.\) Additionally, let \(\lambda_{m}^0=m\pi+(\beta-\alpha)\) and \(\alpha_m^0=1\) denote the eigenvalues and norming constants of the problem \( L(0,0,\alpha,\beta)\). For each fixed \(x\in \left( 0, 1\right]\), the kernel \(K_\alpha(x,t)\) of the transmutation operator \eqref{eqn:TOPKalpha} satisfies the Gelfand-Levitan equation
		\begin{equation}\label{eqn:GL-Kalpha}
			K_\alpha(x,t)+F_\alpha(x,t)+\int_{0}^{x}K_\alpha(x,s)F_\alpha(s,t)ds=0, \quad 0\leq t \leq x \leq 1,
		\end{equation}
		where
		\begin{equation}\label{eqn:GL-Falpha}
			\begin{split}
				F_\alpha(x,t)=\sum_{m=-\infty}^{\infty} 	&\left[ \frac{1}{\alpha_m} \begin{pmatrix}
					\sin \left( \lambda_m x + \alpha \right) \sin \left( \lambda_m t + \alpha \right) & -\sin \left( \lambda_m x + \alpha \right) \cos \left( \lambda_m t + \alpha \right) \\
					-\cos \left( \lambda_m x + \alpha \right) \sin \left( \lambda_m t + \alpha \right) & \cos \left( \lambda_m x + \alpha \right) \cos \left( \lambda_m t + \alpha \right)
				\end{pmatrix} \right.\\
				&\left. - \begin{pmatrix}
					\sin \left( \lambda_m^0 x + \alpha \right) \sin \left( \lambda_m^0 t + \alpha \right) & -\sin \left( \lambda_m^0 x + \alpha \right) \cos \left( \lambda_m^0 t + \alpha \right) \\
					-\cos \left( \lambda_m^0 x + \alpha \right) \sin \left( \lambda_m^0 t + \alpha \right) & \cos \left( \lambda_m^0 x + \alpha \right) \cos \left( \lambda_m^0 t + \alpha \right)
				\end{pmatrix} \right].
			\end{split}
		\end{equation}
	\end{Thm}
	\begin{Rmk}\label{rmk:Falph0}
		Although theorem \ref{thm:GL-KS} is a particular case of theorem \ref{thm:GL-Kalph} we decided to include it to show the following relation among the expressions for \(F_S\) and \(F_\alpha.\)  Note that
		\begin{equation}\label{eqn:Falph0}
			F_\alpha(x,t)=R_\alpha F_{\alpha,0}(x,t) R_\alpha^T.
		\end{equation}
		where \(F_{\alpha,0}(x,t)\) has the same expression as the right-hand side of eq. \eqref{eqn:GL-FS} but with the eigenvalues \(\lambda_{m}\) and norming constants \(\alpha_{m}\) of the problem \(L(p,q,\alpha,\beta)\) instead of those of the problem \(L(p,q,0,0).\) Also, the eigenvalues \(\lambda_{m}^0\) and norming constants \(\alpha_{m}^0\) are those corresponding to the problem \(L(0,0,\alpha,\beta)\) instead of those of \(L(0,0,0,0).\)  Thus, we write  \(F_{\alpha,0}(x,t)\) to be explicit about this dependence of the spectral data and avoid confusion. See also remark \ref{rmk:WCWalph}.
	\end{Rmk}
	
	\section{A method of solution of the Gelfand-Levitan equation}\label{sec:linsys}
	We begin this section by introducing a matrix-valued inner product. The idea to use a matrix-valued inner product appeared naturally during the development of this work. However, we are not aware of a standard reference for it. For that reason, we include some of the properties that we will use in this section. We would like to mention that the concept of a matrix-valued inner product has been proposed in \cite{Xia98,Anto, GinWalden} and recently it has been studied in more detail in \cite{inner}. Let us denote
	\begin{equation}\label{eqn:ip}
		\langle \cdot, \cdot \rangle: L^2([a,b],\mathcal{M}_2)\times L^2([a,b],\mathcal{M}_2) \to \mathcal{M}_2, \quad \langle A_1,A_2\rangle :=\int_{a}^{b} A_1(t)A_2^*(t) \; dt.
	\end{equation}
	
	\begin{Prop}\label{prop:ip}
		Let \(A_1, A_2, A_3 \in L^2([a,b],\mathcal{M}_2),\) \(D\in \mathcal{M}_2\) and \(\lambda \in \mathbb{C}\). The matrix-valued inner product \( \langle \cdot, \cdot \rangle \) satisfies the following properties:
		\begin{enumerate}
			\item \( \langle \lambda A_1+A_2,A_3 \rangle =\lambda \langle A_1,A_3\rangle + \langle A_2,A_3\rangle. \)
			\item \( \langle  DA_1, A_2 \rangle  = D\langle  A_1, A_2 \rangle. \)
			\item \( \langle  A_1, A_2 \rangle  = \langle  A_2, A_1 \rangle^*. \)
			\item \( \left\vert \langle  A_1,A_2 \rangle  \right\vert \leq \Vert A_1 \Vert_{L^2([a,b])} \Vert A_2 \Vert_{L^2([a,b])}. \)
			\item If \( A=\sum_{n=0}^{\infty} A_n \) is a convergent series in \( L^2([a,b], \mathcal{M}_2) \) and \(P \in  L^2([a,b], \mathcal{M}_2) \), then
			\[
			\langle  A, P \rangle  =\sum_{n=0}^{\infty} \langle A_n, P \rangle , \quad 		\langle  P,A \rangle  = \sum_{n=0}^{\infty} \langle P, A_n \rangle .
			\]
		\end{enumerate}
	\end{Prop}
	\begin{proof}
		The proof of the first three properties is straightforward. The fourth property follows from the Cauchy-Schwarz inequality and basic properties of the matrix norm,
		\[
		\left\vert \langle  A_1,A_2 \rangle  \right\vert = \left\vert \int_{a}^{b} A_1(t)A_2^*(t) \; dt \right\vert \leq \int_{a}^{b} \left\vert  A_1(t) \right\vert \left\vert A_2^*(t) \right\vert \; dt \leq \Vert A_1 \Vert_{L^2([a,b])} \Vert A_2 \Vert_{L^2([a,b])},
		\]
		which shows that \( \langle \cdot , A \rangle, \; \langle A , \cdot \rangle : L^2([a,b],\mathcal{M}_2) \to \mathcal{M}_2  \) are continuous. Hence, the fifth property holds.
	\end{proof}
	\begin{Def}
		Let \(A_1, A_2 \in L^2([a,b],\mathcal{M}_2).\) We say that \(A_1,A_2\) are orthogonal if \( \langle A_1,A_2 \rangle={0}_2,\) where \(0_2\) denotes the null matrix of the space \(\mathcal{M}_2\). A set of matrix-valued functions is called an orthogonal set if any two distinct matrix-valued functions in the set are orthogonal. A sequence \( \{A_k\}_{k=0}^\infty \subseteq L^2([a,b],\mathcal{M}_2)\) is called an orthonormal set in \(L^2([a,b],\mathcal{M}_2)\) if
		\[
			\langle A_k, A_j \rangle = \delta_{kj} I, \; k,j =0,1,\ldots,
		\]
		where \(\delta_{kj}=1\) when \(k=j\) and \(\delta_{kj}=0\) otherwise.
	\end{Def}
	\begin{Def}
		An orthonormal sequence  \( \{A_k\}_{k=0}^\infty \subseteq L^2([a,b],\mathcal{M}_2) \) is called an orthonormal basis of \(L^2([a,b],\mathcal{M}_2)\) if for any \(F \in L^2([a,b],\mathcal{M}_2)\) there exist a sequence of constant matrices \(\{F_k\}_{k=0}^\infty \subset \mathcal{M}_2\) such that
		\[
			F=\sum_{k=0}^\infty F_k A_k, \quad F_k = \langle F, A_k \rangle,
		\]
		 and the series converges with respect to the norm \eqref{eqn:norm2}.
	\end{Def}
	\begin{Rmk}
			For more details about the matrix-valued inner product, we invite the reader to consult the reference \cite{inner}. We introduced the matrix-valued inner product to have a concise notation and reduce computations. However, the orthogonal sequences considered in our work do not need such a detailed analysis. For instance, consider the sequence \(A_n^C(x,t)\) and take \(F=(F_{ij})_{i,j=1}^2\in L^2([0,x], \mathcal{M}_2).\) Then,
		\[
		F(t)A_n^C(x,t) = \begin{pmatrix}
			F_{11}(t) & F_{12}(t) \\ F_{21}(t) & F_{22}(t)
		\end{pmatrix} \begin{pmatrix}
		P_{2n}(\frac{t}{x}) & 0 \\ 0 & P_{2n+1}(\frac{t}{x})
		\end{pmatrix} = \begin{pmatrix}
		F_{11}(t)P_{2n}(\frac{t}{x}) & F_{12}(t)P_{2n+1}(\frac{t}{x}) \\ F_{21}(t)P_{2n}(\frac{t}{x}) & F_{22}(t)P_{2n+1}(\frac{t}{x})
		\end{pmatrix}.
		\]
		Hence, matrix Fourier-Legendre series representations such as those in \eqref{eqn:KCKS} are just Fourier-Legendre series entry-wise.
	\end{Rmk}
	\subsection{Reduction of the Gelfand-Levitan equation to a linear system of algebraic equations}
	In this section we obtain a linear system of algebraic equations for the Gelfand-Levitan equation for the kernels \(K_C,\) \(K_S\) and \(K_\alpha.\) From now on, \( \langle \cdot, \cdot \rangle: L^2([0,x],\mathcal{M}_2)\times L^2([0,x],\mathcal{M}_2) \to \mathcal{M}_2 \) denotes the matrix-valued inner product on the interval \([0,x].\)\\
	
	Let us start with the kernel \(K_C.\) Since \(F_C\) is real-valued, we observe that the Gelfand-Levitan equation can be rewritten as
	\begin{equation}\label{eqn:GL-KCip}
		K_C(x,t)+\left \langle K_C(x,\cdot), F_C(\cdot,t)^T \right \rangle = - F_C(x,t).
	\end{equation}
	Since \(F_C \in L^2([0,x]\times [0,x], \mathcal{M}_2)\), it is a Hilbert-Schmidt kernel. Hence, it follows that the linear operator
	\begin{equation}
		\mathcal{F}_C: L^2([0,x],\mathcal{M}_2) \to L^2([0,x],\mathcal{M}_2), \quad X \mapsto \mathcal{F}_C[X](t):=\left \langle X, F_C(\cdot,t)^T \right \rangle,
	\end{equation}
	is a compact linear operator and the Gelfand-Levitan equation in terms of \( \mathcal{F}_C\) can be written
	\begin{equation}\label{eqn:GL-IF}
		(\mathcal{I}+\mathcal{F}_C)[K_C(x,\cdot)]=-F_C(x,\cdot),
	\end{equation}
	where \(\mathcal{I}\) denotes the identity operator. This observation is important to establish some stability results further down below. As a basis of \(L^2([0,x],\mathcal{M}_2)\) we choose the \(\mathcal{M}_2\)-valued functions \(\{A_n^C\}_{n=0}^\infty\) defined in \eqref{eqn:Kcn}, which is a complete orthogonal system with respect to the matrix-valued inner product \eqref{eqn:ip}.
	Substitution of the Fourier-Legendre series representation \eqref{eqn:KCKS} for \(K_C\) in \eqref{eqn:GL-KCip} leads to
	\begin{equation}\label{eqn:GLfn}
		\sum_{n=0}^{\infty} \frac{1}{x} K^C_n(x) \left[ A_n^C\left( x,t \right) + f^C_n(x,t) \right]=-\frac{1}{2}F_C(x,t), \quad 	 f^C_n(x,t):=\left \langle A_n^C(x, \cdot), F_C(\cdot,t)^T \right \rangle.
	\end{equation}
	where we have used that the Fourier-Legendre series for \(K^C\) is convergent, see \cite{RoqTorNSBF} for details.
	We know that \(f^C_n(x,\cdot) \in L^2([0,x],\mathcal{M}_2).\) Furthermore, \( F_C(\cdot,t) \in L^2([0,x], \mathcal{M}_2)\). Using eq. \eqref{eqn:GL-FC}  and the continuity property of the inner product (see proposition \ref{prop:ip}) we have that
	\begin{equation}
		f^C_n(x,t) = \sum_{m=-\infty}^\infty \left \langle A_n^C(x, \cdot), \frac{1}{\alpha_m} M_C(\lambda_{m},\cdot,t)^T - M_C(\lambda_{m}^0,\cdot,t)^T \right \rangle.
	\end{equation}
	From the definition \eqref{eqn:Mlambxt} of \(M_C(\lambda,x,t)\)  it follows that
	\begin{equation}
		\left \langle A_n^C(x, \cdot), M_C(\lambda,\cdot,t)^T  \right \rangle = \int_{0}^{x} \begin{pmatrix}
			P_{2n}\left( \frac{s}{x} \right) \cos \lambda s \cos \lambda t & P_{2n}\left( \frac{s}{x} \right) \cos \lambda s \sin \lambda t \\
			P_{2n+1}\left( \frac{s}{x} \right) \sin \lambda s \cos \lambda t & P_{2n+1}\left( \frac{s}{x} \right) \sin \lambda s \sin \lambda t
		\end{pmatrix} \; ds.
	\end{equation}
	Thus, using the formulas \eqref{eqn:intPncs} we get
	\begin{equation}\label{eqn:fCn}
		f^C_n(x,t)=(-1)^nx\sum_{m=-\infty}^{\infty}\left[ \frac{1}{\alpha_m} V^C_n(\lambda_m,x,t)-V^C_n(\lambda_m^0,x,t) \right],
	\end{equation}
	where
	\begin{equation}
		V^C_n(\lambda,x,t):=\begin{pmatrix}  j_{2n}(\lambda x) \cos \lambda t &  j_{2n}(\lambda x) \sin \lambda t \\
			j_{2n+1}(\lambda x) \cos \lambda t &  j_{2n+1}(\lambda x) \sin \lambda t \end{pmatrix},
	\end{equation}
	Now, we multiply both sides of the equality \eqref{eqn:GLfn} by \( A^C_k(x,t) \) from the right, and we integrate from zero to \(x\) with respect to the variable \(t\). For the right-hand side we obtain
	\begin{equation}
		-\frac{1}{2}f^C_k(x,x)^T=\left \langle -\frac{1}{2}F_C(x,\cdot), A^C_k(x,\cdot) \right \rangle=(-1)^{k+1} \frac{1}{2}x \sum_{m=-\infty}^{\infty} \left[ \frac{1}{\alpha_m} V^C_k(\lambda_m, x, x)^T- V^C_k(\lambda_m^0, x, x)^T \right].
	\end{equation}
	After similar computations for the left-hand side, we obtain the following theorem.
	\begin{Thm}
		The coefficients \(K_k^C\) satisfy the system of linear algebraic equations
		\begin{equation}\label{eqn:main-system}
			K^C_k(x)D^C_k+\sum_{n=0}^{\infty} K^C_n(x)W^C_{nk}(x)=-\frac{1}{2}f^C_k(x,x)^T, \quad k=0,1,\ldots,
		\end{equation}
		where
		\begin{equation}\label{eqn:WCnk}
			\begin{split}
				W^C_{nk}(x):=(-1)^{n+k}x\sum_{m=-\infty}^{\infty}&\left[ \frac{1}{\alpha_m}\begin{pmatrix} j_{2n}(\lambda_m x) j_{2k}(\lambda_m x) & j_{2n}(\lambda_m x)j_{2k+1}(\lambda_m x)  \\
					j_{2n+1}(\lambda_m x) j_{2k}(\lambda_m x) & j_{2n+1}(\lambda_m x) j_{2k+1}(\lambda_m x) \end{pmatrix}\right.  \\
				&\left. - \begin{pmatrix} j_{2n}(\lambda_m^0 x) j_{2k}(\lambda_m^0 x) & j_{2n}(\lambda_m^0 x)j_{2k+1}(\lambda_m^0 x) \\
					j_{2n+1}(\lambda_m^0 x) j_{2k}(\lambda_m^0 x) & j_{2n+1}(\lambda_m^0 x) j_{2k+1}(\lambda_m^0 x) \end{pmatrix} \right] .
			\end{split}
		\end{equation}
		and
		\begin{equation}
			D^C_k:=\begin{pmatrix}
				\frac{1}{4k+1} & 0 \\ 0 & \frac{1}{4k+3}
			\end{pmatrix}.
		\end{equation}
	\end{Thm}
	We point out that a truncation of the linear system of equations coincides with the application of the Bubnov-Galerkin method after a suitable normalization of the coefficients. Consider the set of functions
	\begin{equation}
		p_k(t):=\frac{\sqrt{2k+1}}{\sqrt{x}}P_k\left(\frac{t}{x} \right), \quad k=0,1,2,\ldots,
	\end{equation}
	and consider the set of matrix-valued functions
	\begin{equation}
		\widehat{A}_k^C:=\begin{pmatrix}
			p_{2k}(t) & 0 \\ 0 & p_{2k+1}(t)
		\end{pmatrix}.
	\end{equation}
	Then, \(\widehat{A}_k^C, \; k=0,1,2,\ldots\) form an orthonormal basis of \(L^2([0,x], \mathcal{M}_2)\) with respect to the inner product \eqref{eqn:ip} with matrix-valued coefficients. The matrix coefficients \(\frac{K_k^C (D_k^C)^{1/2}}{\sqrt{x}}\) are the Fourier coefficients of the kernel \(K_C(x,t)\) with respect to the basis \( \{ \widehat{A}_k^C \}_{k=0}^\infty.\) Multiplying \eqref{eqn:main-system} from the right by \( \frac{ (D^C_k)^{-1/2}}{\sqrt{x}}\), the system of equations can be rewritten as
	\begin{equation}\label{eqn:normalized-system}
		\left( \frac{K_k^C (D^C_k)^{1/2}}{\sqrt{x}} \right) + \sum_{n=0}^\infty \left( \frac{K_n^C (D^C_n)^{1/2}}{\sqrt{x}} \right) \left[ (D^C_n)^{-1/2}W^C_{nk}(x) (D^C_k)^{-1/2}  \right] = -\frac{1}{2} \left[ f^C_k(x,x)^T \frac{(D^C_k)^{-1/2}}{\sqrt{x}}\right].
	\end{equation}
	A thorough investigation of the existence and stability of the solution for the truncated system in the case of the Schrödinger operator can be found in \cite{KravTorbaInv}. The analysis for our system \eqref{eqn:normalized-system} is rather similar and therefore we do not repeat it. The reader is highly encouraged to read \cite[section 3.5]{KravTorbaInv} for details, since the arguments appearing there just have to be applied entry-wise. See also appendix \ref{section:appendix-bg} for a brief description of the Bubnov-Galerkin method and a well-known result regarding its stability. Nonetheless, let us state a summarizing proposition for the sake of clarity, which is analogous to \cite[proposition 3.9]{KravTorbaInv}. Let us set
	\begin{equation}
		\xi_k := \frac{K^C_k(x)(D_k^C)^{1/2}}{\sqrt{x}}, \quad a_{nk}:= (D^C_n)^{-1/2}W^C_{nk}(x) (D^C_k)^{-1/2}, \quad b_k:= -\frac{1}{2} \left[ f^C_k(x,x)^T \frac{(D^C_k)^{-1/2}}{\sqrt{x}}\right].
	\end{equation}
	\begin{Prop}
		For a sufficiently large \(N\), the truncated system
		\begin{equation}
			\xi_k + \sum_{n=0}^N \xi_k a_{nk} = b_k, \quad k=0,1,\ldots, N,
		\end{equation}
		has a unique solution \(U_N=\{ \xi_k^N\}_{k=0}^N\). Moreover, for small changes in the coefficients \(a_{nk}\) and \(b_k\) (e.g., due to the truncation of \eqref{eqn:fCn} and \eqref{eqn:WCnk}), the condition number of the truncated system is uniformly bounded with respect to \(N\), and the solution \(U_N\) is stable.
	\end{Prop}
%	Similarly, we can reduce the Gelfand-Levitan equation for the kernel \(K_S\) to a linear system of algebraic equations.
%	\begin{Thm}
%		The coefficients \(K_k^S\) of the kernel \(K_S(x,t)\) \eqref{eqn:KCKS} satisfy the linear system of algebraic equations
%		\begin{equation}\label{eqn:system-KS}
%			K^S_k(x)D^S_k+\sum_{n=0}^{\infty} K^S_n(x)W^S_{nk}(x)=-\frac{1}{2}f^S_k(x,x)^T, \quad k=0,1,\ldots,
%		\end{equation}
%		where
%		\begin{equation}
%			\begin{split}
%				W^S_{nk}(x):=(-1)^{n+k}x\sum_{m=-\infty}^{\infty}&\left[ \frac{1}{\alpha_m}\begin{pmatrix} j_{2n+1}(\lambda_m x) j_{2k+1}(\lambda_m x) & -j_{2n+1}(\lambda_m x)j_{2k}(\lambda_m x)  \\
%					-j_{2n}(\lambda_m x) j_{2k+1}(\lambda_m x) & j_{2n}(\lambda_m x) j_{2k}(\lambda_m x) \end{pmatrix}\right.  \\
%				&\left. - \begin{pmatrix} j_{2n+1}(\lambda_m^0 x) j_{2k+1}(\lambda_m^0 x) & -j_{2n+1}(\lambda_m^0 x)j_{2k}(\lambda_m^0 x) \\
%					-j_{2n}(\lambda_m^0 x) j_{2k+1}(\lambda_m^0 x) & j_{2n}(\lambda_m^0 x) j_{2k}(\lambda_m^0 x) \end{pmatrix} \right] ,
%			\end{split}
%		\end{equation}
%		\begin{equation}
%			f_k^S(x,t)=(-1)^k x \sum_{m=-\infty}^\infty \left[\frac{1}{\alpha_m}V_k^S(\lambda_{m},x,t)-V_k^S(\lambda_{m}^0,x,t)\right],
%		\end{equation}
%		and
%		\begin{equation}
%			V_k^S(\lambda,x,t) = \begin{pmatrix}
%				j_{2k+1}(\lambda x) \sin \lambda t & -j_{2k+1}(\lambda x) \cos \lambda t \\
%				-j_{2k}(\lambda x) \sin \lambda t & j_{2k}(\lambda x) \cos \lambda t
%			\end{pmatrix},
%			\quad
%			D^S_k:=\begin{pmatrix}
%				\frac{1}{4k+3}	 & 0 \\ 0 & \frac{1}{4k+1}
%			\end{pmatrix}.
%		\end{equation}
%	\end{Thm}
	Similarly, thanks to eq. \eqref{eqn:Falph0} (see remark \ref{rmk:Falph0}) we can rewrite eq. \eqref{eqn:GL-Kalpha} as
	\begin{equation}
		K_\alpha(x,t)+\int_{0}^{x}K_\alpha(x,s)R_\alpha F_{\alpha,0}(s,t) R_\alpha^T\, ds=-R_\alpha F_{\alpha,0}(x,t)R_\alpha^T.
	\end{equation}
	Substitution of \eqref{eqn:KalphaASR} in the equation above together with right-multiplication by \(R_\alpha\) yields
	\begin{equation}
		\sum_{n=0}^\infty \frac{1}{x}K_n^\alpha(x)\left[A_n^S(x,t)+\int_{0}^{x}A_n^S(x,s)F_{\alpha,0}(s,t) \, ds\right]=-\frac{1}{2}R_\alpha F_{\alpha,0}(x,t).
	\end{equation}
	From here, we can follow the same procedure as before to obtain a linear system of algebraic equations for the coefficients \(K_n^\alpha.\)
	\begin{Thm}
		The coefficients \(K_k^\alpha\) of the kernel \(K_\alpha(x,t)\) \eqref{eqn:KalphaASR} satisfy the linear system of algebraic equations
		\begin{equation}\label{eqn:system-Kalpha}
			K^\alpha_k(x)D^S_k+\sum_{n=0}^{\infty} K^\alpha_n(x)W^\alpha_{nk}(x)=-\frac{1}{2}R_\alpha f^\alpha_k(x,x)^T, \quad k=0,1,\ldots,
		\end{equation}
		where
		\begin{equation}
			\begin{split}
				W^\alpha_{nk}(x):=(-1)^{n+k}x\sum_{m=-\infty}^{\infty}&\left[ \frac{1}{\alpha_m}\begin{pmatrix} j_{2n+1}(\lambda_m x) j_{2k+1}(\lambda_m x) & -j_{2n+1}(\lambda_m x)j_{2k}(\lambda_m x)  \\
					-j_{2n}(\lambda_m x) j_{2k+1}(\lambda_m x) & j_{2n}(\lambda_m x) j_{2k}(\lambda_m x) \end{pmatrix}\right.  \\
				&\left. - \begin{pmatrix} j_{2n+1}(\lambda_m^0 x) j_{2k+1}(\lambda_m^0 x) & -j_{2n+1}(\lambda_m^0 x)j_{2k}(\lambda_m^0 x) \\
					-j_{2n}(\lambda_m^0 x) j_{2k+1}(\lambda_m^0 x) & j_{2n}(\lambda_m^0 x) j_{2k}(\lambda_m^0 x) \end{pmatrix} \right] ,
			\end{split}
		\end{equation}
		\begin{equation}
			f_k^\alpha(x,t)=(-1)^k x \sum_{m=-\infty}^\infty \left[\frac{1}{\alpha_m}V_k^\alpha(\lambda_{m},x,t)-V_k^\alpha(\lambda_{m}^0,x,t)\right],
		\end{equation}
		and
		\begin{equation}
			V_k^\alpha(\lambda,x,t) = \begin{pmatrix}
				j_{2n+1}(\lambda x) \sin \lambda t & -j_{2n+1}(\lambda x) \cos \lambda t \\
				-j_{2n}(\lambda x) \sin \lambda t & j_{2n}(\lambda x) \cos \lambda t
			\end{pmatrix},
			\quad
			D^S_k:=\begin{pmatrix}
				\frac{1}{4k+3}	 & 0 \\ 0 & \frac{1}{4k+1}
			\end{pmatrix}.
		\end{equation}
	\end{Thm}
	\begin{Rmk}\label{rmk:WCWalph}
%		Notice that the expressions for \(W^\alpha_{nk}(x)\) and \(f_k^\alpha(x,t) \) are the same as those for \(W^S_{nk}(x)\) and \(f_k^S(x,t)\) but for different spectral data. This is why we introduced other names for the coefficients to avoid confusion. See also remark \ref{rmk:Falph0}. Furthermore, f
		For any matrix \(A=(a_{ij})\in \mathcal{M}_2\), conjugation of \(A\) by \(B\) is as follows,
		\begin{equation}
			\begin{pmatrix}
				0 & 1 \\ -1 & 0
			\end{pmatrix} \begin{pmatrix}
				a_{11} & a_{12} \\ a_{21} & a_{22}
			\end{pmatrix}\begin{pmatrix}
				0 & 1 \\ -1 & 0
			\end{pmatrix}^T = \begin{pmatrix}
				a_{22} & -a_{21} \\ -a_{12} & a_{11}
			\end{pmatrix}.
		\end{equation}
		Thus, the expressions for \(W^C_{nk}(x)\) and \( f_k^C(x,t)\) are related to the expressions for \(W^\alpha_{nk}(x)\) and \(f_k^\alpha(x,t) \) via conjugation by \(B\) but for different spectral data.
	\end{Rmk}
	\subsection{Two possibilities for the recovery of the potential} \label{sec:recovery}
	In this section we present two different methods to recover the potential \(Q\). We focus on the system of equations for the coefficients \(K_C\). Other cases can be treated similarly. In practice, we can truncate system \eqref{eqn:main-system} and obtain a finite number of approximate coefficients \( \tilde{K}^C_n,\)  \(n=0,\ldots,N.\) We then have an approximation \(\tilde{K}^N_C\) of \(K_C\) via the \(N\)-th partial sum
	\begin{equation}
		\tilde{K}^N_C(x,t):= \sum_{n=0}^N \tilde{K}^C_n(x)A^C_n(x,t).
	\end{equation}
	Substitution of \(\tilde{K}^N_C\) for \(K_C\) in eq. \eqref{eqn:Kalphxx} and the fact that the Legendre polynomials satisfy \(P_n(1)=1\) leads to an approximate recovery \(Q^N\) of \(Q\) given by
	\begin{equation}\label{eqn:QGoursat}
		Q(x) \approx Q^N(x) := \tilde{K}^N_C(x,x)B-B\tilde{K}^N_C(x,x) = \sum_{n=0}^N \left[ \tilde{K}^C_n(x)B-B\tilde{K}^C_n(x) \right],
	\end{equation}
	and we know that \(Q^N \to Q\) as \(N \to \infty\), due to the stability of the process. One might think that in order to get a good accuracy in the approximate recovery of \(Q\), we must compute a large number of coefficients \( K^C_n \). The decay rate of the coefficients shows that at least for a smooth potential, a small number of coefficients are needed. The following result regarding the decay rate of the coefficients \(K_n^C\) in terms of the smoothness of \(Q\) is stated without proof as it is only a slight modification of \cite[Cor. 4.9]{RoqTorNSBF}.
	\begin{Prop}
		Let \(Q \in W^{p+1,2}([0,1]), p\geq 0\). Then,
		\begin{equation}
			\vert K^C_n(x) \vert \leq \frac{C(p,b,Q)x^{p+3/2}}{n^{p+1/2}}, \quad n>p+2,
		\end{equation}
		and \(C(p,b,Q)\) is a constant that depends solely on \(p,\, b\) and \(Q\).
	\end{Prop}
	Nonetheless, it is possible to recover the potential \(Q\) only from the first coefficient \( K_0^C\). %The following proposition gives us another possibility for the recovery of \(Q\).
	\begin{Prop}\label{prop:pqab}
		Let \(Q\) be a real matrix-valued potential and let \( v(x):= (a(x), b(x))^T \) be the vector which equals the first column of \( K^C_0(x) \). Then it is possible to recover \( Q \) from the following equality
		\begin{equation}\label{eqn:pqab}
			\begin{pmatrix} p \\ q \end{pmatrix} = \frac{1}{4b^2+(1+2a)^2}\begin{pmatrix} -2b & -(1+2a) \\ 1+2a & -2b \end{pmatrix} \begin{pmatrix} 2a' \\ 2b' \end{pmatrix}=\frac{2}{4b^2+(1+2a)^2}\begin{pmatrix} - 2(ab)'-b' \\  (a^2)'-(b^2)'+a' \end{pmatrix}.
		\end{equation}
	\end{Prop}
	\begin{proof}
		For \( \lambda =0 \), using eq. \eqref{eqn:nsbfCS} we have that
		\begin{equation}
			C(0,x)= C_0(0, x)+ 2 v(x), \quad  C_0(0, x) = \begin{pmatrix}
				1 \\ 0
			\end{pmatrix},
		\end{equation}
		and we can see that \(v\) is real-valued since \(C(0,x)\) is. Therefore its entries \(a,\, b\) are real. Then,
		\[
		2v'(x)=C(0,x)'=BQC(0, x)= BQ(C_0(0, x)+2v(x)),
		\]
		From here, we get that
		\[
		\begin{pmatrix} q & -p \\ -p & -q \end{pmatrix} \begin{pmatrix} 1+2a \\ 2b \end{pmatrix} = \begin{pmatrix} 2a' \\ 2b' \end{pmatrix},
		\]
		which implies that
		\[
		\begin{pmatrix} -2b & 1 + 2a \\ -(1+2a) & -2b \end{pmatrix}\begin{pmatrix} p \\ q \end{pmatrix}= \begin{pmatrix} 2a' \\ 2b' \end{pmatrix}.
		\]
		The determinant of the system is \( \Delta = 4b^2+(1+2a)^2 \). We claim that \( \Delta >0 \). Indeed, let us assume that there exists \( x_0 \in (0,1] \) such that \( \Delta(x_0)= 0 \) (\( \Delta(0)=1 \)). Then \( a(x_0)= -\frac{1}{2} \) and \( b (x_0)= 0 \). This implies \( C(0,x_0)= (0, 0)^T \). Due to the uniqueness of the solution, \( C(0,x)\equiv (0,0)^T \) which is impossible since \( C(0,0)=(1,0)^T \).
	\end{proof}
	Recovery of a potential from the first coefficient of an NSBF series has been studied in several inverse problems for the Schrödinger equation, see \cite{KravTorbaInv, krav-icp, kravchenko-gb, krav-cp}
	and the references therein. As it is shown in those works, only a handful of equations are needed to get a good recovery of the potential. This is a remarkable property of the NSBF representation of solutions that we have been able to generalize to the Dirac system. \\
	
	Similarly, we can recover the potential \(Q\) from the first coefficient \(K_0^S\) and \(K_0^\alpha.\)
	\begin{Prop}
		\begin{enumerate}
			\item 	Let \( w(x):= (c(x), d(x))^T \) be the vector which equals the second column of \( K^S_0(x) \). Then it is possible to recover \( Q \) from the following equality
			\begin{equation}\label{eqn:pqcd}
				\begin{pmatrix} p \\ q \end{pmatrix} = \frac{1}{(1+2d)^2+4c^2}\begin{pmatrix} -(1+2d) & -2c \\ 2c & -(1+2d) \end{pmatrix} \begin{pmatrix} 2c' \\ 2d' \end{pmatrix}.
			\end{equation}
			\item 	Let \( w_\alpha(x):= (c_\alpha(x), d_\alpha(x))^T \) be the vector which equals the second column of \( K^\alpha_0(x) \). Then it is possible to recover \( Q \) from the following equality
			\begin{equation}\label{eqn:pqcalpha}
				\begin{pmatrix} p \\ q \end{pmatrix} = \frac{1}{(\cos \alpha+2d_\alpha)^2+(\sin \alpha - 2c_\alpha)^2}\begin{pmatrix} -(\cos \alpha+2d_\alpha) & \sin \alpha -2c_\alpha \\ -\sin\alpha +2 c_\alpha & -(\cos \alpha+2d_\alpha) \end{pmatrix} \begin{pmatrix} 2c_\alpha' \\ 2d_\alpha' \end{pmatrix}.
			\end{equation}
		\end{enumerate}
	\end{Prop}
	\begin{proof}
		The proof is analogous to the proof of proposition \ref{prop:pqab}, using the NSBF representation for \(S(\lambda,x)\) and \(\varphi_\alpha(\lambda,x)\) respectively.
	\end{proof}
	\section{Numerical implementation}\label{sec:numerical}
	
	In this section, we present several numerical examples. The implementation of the ideas presented in the previous sections is mostly straightforward, and thus we only discuss a few details. The procedure is considerably similar to that of \cite[Section 5]{KravTorbaInv}, where a thorough description of the numerical implementation is presented, and the reader is encouraged to consult more details there. In all of our examples, \(\alpha = -\pi/2\) and it is assumed to be known. Regarding the parameter \(\beta\), we analyze two cases, see the statement of the inverse problems \ref{prob:bothbc} and \ref{prob:onebc}. First, we assume \(\beta \) to be known and equal to zero. Then, we assume \(\beta\) to be unknown and we recover it numerically. Other cases for \(\alpha\) and \(\beta\) can be treated similarly. All the code was implemented in Matlab 2024b. The main steps of the algorithm are as follows:
	\begin{enumerate}
		\item Given a finite set of spectral data
		\begin{equation}
			\{ \lambda_{m} \}_{m=-M_1}^{M_2}, 	\{ \alpha_{m} \}_{m=-M_1}^{M_2}, \qquad M_1,M_2>0,
		\end{equation}
		we consider a truncation of the system \eqref{eqn:main-system} with a number \(N_C\) of equations and on a finite mesh of points \( 0=x_0<x_1<\ldots < x_p=1\). We consider a uniform mesh in the implementation for simplicity with 80--100 points. We use \(N_C=10\) in all of our examples.  However, we would like to point out that such a number of equations was necessary only for recovery of the matrix potential with the accuracy of \(10^{-8}\) in Example \ref{ex:sincos}. As few as 3--5 equations were sufficient for all other examples.
		\item In practice only a handful of eigenvalues and norming constants are known. However, the series for the coefficients \(W^C_{nk}\)  \eqref{eqn:WCnk} and \(f_k^C(x,x)\) \eqref{eqn:fCn}  converge slowly and a large number of spectral data are needed to have a good accuracy, see remark \ref{rmk:nonconvergent}. This is why, instead of
		approximating the coefficients \(W^C_{nk}\) \eqref{eqn:WCnk} of the system by
		\begin{equation}
			\begin{split}
				W^C_{nk}(x)\approx (-1)^{n+k}x\sum_{m=-M_1}^{M_2}&\left[ \frac{1}{\alpha_m}\begin{pmatrix} j_{2n}(\lambda_m x) j_{2k}(\lambda_m x) & j_{2n}(\lambda_m x)j_{2k+1}(\lambda_m x)  \\
					j_{2n+1}(\lambda_m x) j_{2k}(\lambda_m x) & j_{2n+1}(\lambda_m x) j_{2k+1}(\lambda_m x) \end{pmatrix}\right.  \\
				&\left. - \begin{pmatrix} j_{2n}(\lambda_m^0 x) j_{2k}(\lambda_m^0 x) & j_{2n}(\lambda_m^0 x)j_{2k+1}(\lambda_m^0 x) \\
					j_{2n+1}(\lambda_m^0 x) j_{2k}(\lambda_m^0 x) & j_{2n+1}(\lambda_m^0 x) j_{2k+1}(\lambda_m^0 x) \end{pmatrix} \right],
			\end{split}
		\end{equation}
		we  follow the strategy presented in \cite[section 3.3]{KravTorbaInv} with minor modifications to improve the accuracy of the approximation: we generate a large number of asymptotic spectral data using asymptotic formulas \eqref{eqn:asym-eig} and \eqref{eqn:asym-normingc}. To be more specific, if we assume \(\beta \) to be known we can recover approximately constants \(c_j\) and \(a_j\) for \(1\leq j \leq K\) (where \(K\) depends on the smoothness of \(Q\)) from the best fit problem
		\begin{equation}\label{eqn:bestfit}
			\sum_{j=1}^{K} \frac{c_j}{m^j}=\lambda_{m}-(m\pi+(\beta-\alpha)), \qquad \sum_{j=1}^{K} \frac{a_j}{m^j}=\alpha_{m}-1.
		\end{equation}
		In the equation above \(m\) ranges in the lower half of the negative indices and the upper half of the positive indices of the spectral data. Since the smoothness of the potential is unknown beforehand, the parameter \(K\) is chosen such that the least squares error of the fit for \(K\) offers a significant improvement compared to the least squares error for \(K-1\).
		In the case where \( \beta \) is unknown, this parameter can be recovered approximately too from the best fit problem
		\begin{equation}
			\beta+\sum_{j=1}^{K} \frac{c_j}{m^j}=\lambda_{m}-(m\pi-\alpha).
		\end{equation}
		\item As already described in section \ref{sec:recovery} there are two possibilities for recovering the potential matrix \(Q.\) The first one, given by \eqref{eqn:QGoursat} is straightforward to implement. For the second one, namely, the recovery of the potential from the first coefficient \(K_0^C\) and formula \eqref{eqn:pqab} we interpolate the values for the first column of \(K_0^C\) using a 6th order spline and Matlab's function \texttt{fnder} to perform the numerical differentiation. As we can see from example \ref{ex:example1} the method based on spline differentiation works significantly better.
	\end{enumerate}
	\begin{Rmk}
		In all of the numerical examples, we consider \(M_1=M_2\) to benefit from the convergence improvement mentioned in remark \ref{rmk:nonconvergent}, i.e., the amount of positive and negative spectral data is the same. Thus, we have an odd number of spectral data in all of the examples. For instance, when we say that we used 201 spectral data, we are considering the eigenvalues and norming constants indexed from -100 to 100, and so on. The asymptotic spectral data is also generated to satisfy this condition.
	\end{Rmk}
	We would like to mention that the most time consuming part of the proposed algorithm is the computation of the spherical Bessel functions. For that reason, they have been computed recursively, as explained in \cite[p. 23]{KravTorbaInv}. Even with such optimization, the computation of the spherical Bessel functions constitutes at least 80\% of the execution time. Therefore, for practical applications where both boundary conditions are fixed and known, storage of the terms \(j_n(\lambda_m^0 x)\) in memory should result in a considerable improvement of the execution time.  In all of our examples, even those where we assume both \(\alpha\) and \(\beta\) to be known the computation time includes the calculation of all the expressions, including the terms involving \(\lambda_{m}^0\), just as if \(\beta\) was not known beforehand.  \\
	
	The exact spectral data has been generated using the procedure proposed in \cite{RoqTorNSBF}, which we use to obtain the eigenvalues and eigenfunctions. The norming constants are then obtained using numerical integration with a Newton-C\^otes 6-point formula. All the computations were performed on a laptop with an 11th Gen Intel Core i7-1185G7 CPU and 32 GB of RAM.
	\begin{Ex}\label{ex:example1}
		Consider the spectral problem \(L(p,q,-\pi/2,0)\) with
		\begin{equation}\label{eqn:example1}
			p(x)=-\frac{(x+1)}{2} \cos \left( \frac{x(x-2)}{2}\right), \quad q(x)=\frac{(x+1)}{2} \sin \left( \frac{x(x-2)}{2}\right).
		\end{equation}
		In figure \ref{fig:Ex1exact}, we present the absolute error of the recovery of the entries \(p\) and \(q\) using 201, 1001, and 2001 exact spectral data pairs. On the left of the figure we present the absolute error of the method based on spline differentiation and formula \eqref{eqn:pqab}. On the right, we present the absolute error of the method based on the Goursat characteristic equation \eqref{eqn:QGoursat}. Although the recovery of \(q\) has comparable accuracy, for \(p\) the absolute error of the recovery is significantly better using spline differentiation. The execution time using 2001 spectral data is around 0.6 seconds. \\
		
		As we have already mentioned, in practice only a small amount of spectral data is available. In figure \ref{fig:Ex1asym} we show the absolute error of the recovery based on spline differentiation using 2001 exact eigenvalues with no asymptotic eigenvalues (AEVs). This is compared to the absolute error of the recovery using 51 and then only 11 exact eigenvalues and generating enough AEVs to have 10,001 spectral data in total indexed from -5000 to 5000. The execution time of the algorithm using 10,001 eigenpairs is around 2.2 seconds. As we can observe, a small number of spectral data is enough to recover the matrix potential. The value of the parameter \(K\) appearing in \eqref{eqn:bestfit} for the case of 11 exact eigenvalues was 5, whereas for 51 exact eigenvalues was 8.
		\begin{figure}
			\centering
			\includegraphics[width=5in,height=4in]{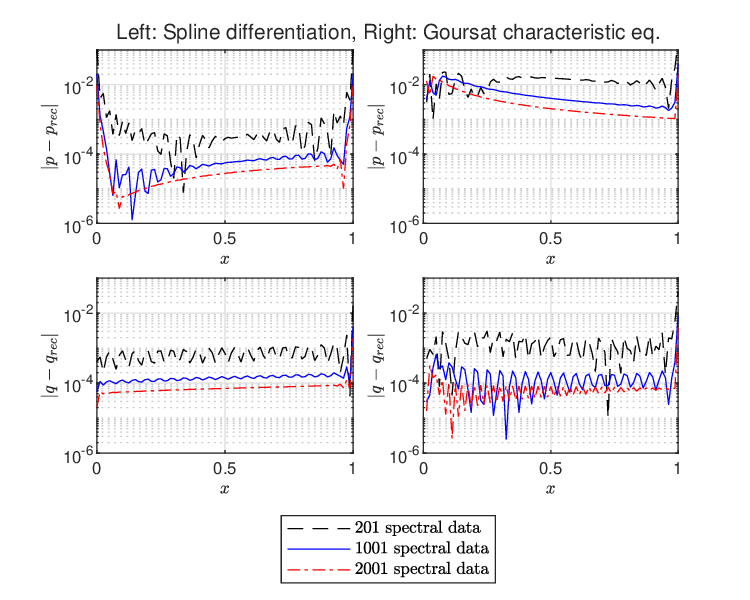}
			\caption{Absolute error of the method based on spline differentiation (left) vs the method based on the Goursat characteristic equation (right) for example \ref{ex:example1}.}
			\label{fig:Ex1exact}
		\end{figure}
		\begin{figure}
			\centering
			\includegraphics[width=5in,height=2.4in]{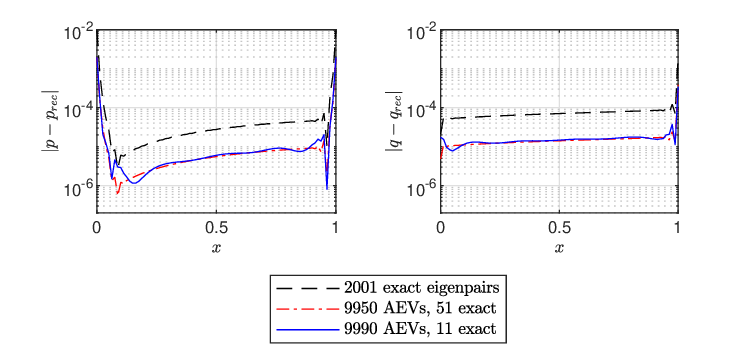}
			\caption{Absolute error of the recovery of the potential \eqref{eqn:example1} based on spline differentiation using 2001 exact eigenvalues with no asymptotic eigenvalues (AEVs). This is compared to the absolute error of the recovery using 51 and 11 exact eigenvalues and generating enough AEVs to have 10,001 spectral data in total. }
			\label{fig:Ex1asym}
		\end{figure}
	\end{Ex}
	
	\begin{Ex}
		Consider the boundary value problem \(L(p,q,-\pi/2,0)\) with
		\begin{equation}\label{eqn:example2}
			p(x)=\left\vert 3 - \left\vert 9x^2-3 \right\vert \right\vert , \quad q(x)=-\left\vert x-\frac{1}{2}\right\vert.
		\end{equation}
		Then, \(Q \in W^{1,2}([0,1],\mathcal{M}_2)\). In figure \ref{fig:Ex2} we show the absolute error of the recovery based on spline differentiation using 21 exact eigenvalues and generating enough AEVs to have 5,001 spectral data in total indexed from -2500 to 2500. The execution time of the algorithm is 1.3 seconds approximately. We would like to point out that the parameter \(K\) appearing in \eqref{eqn:bestfit} is found to be 1 as we would expect.
		\begin{figure}
			\centering
			\includegraphics[width=5in,height=2in]{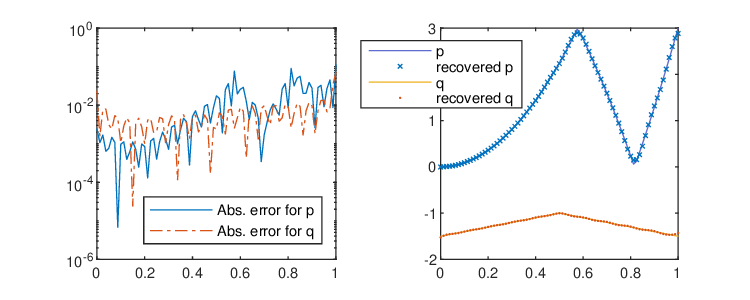}
			\caption{Left: Absolute error of the recovery of the potential \eqref{eqn:example2} based on spline differentiation using 21 exact eigenvalues and 4980 AEVs to have 5,001 spectral data in total.  Right: Exact and recovered potential.}
			\label{fig:Ex2}
		\end{figure}
	\end{Ex}
	\begin{Ex}
		For this example we have taken inspiration in the excellent work \cite{Trefethen}. Consider the boundary value problem \(L(p,q,-\pi/2,0)\) with a sawtooth potential entry for \(Q\),
		\begin{equation}\label{eqn:example3}
			p(x)=\left\vert 3 - \left\vert 9x^2-3 \right\vert \right\vert , \quad q(x)=-10 \int_{0}^x \mathrm{sign} \left( \sin \left(\frac{10 \pi x}{4-\pi x}\right)\right) \, dx.
		\end{equation}
		Then, \(Q \in W^{1,2}([0,1],\mathcal{M}_2)\). In figure \ref{fig:Ex3} we show the absolute error of the recovery based on spline differentiation using 41 exact eigenvalues and generating enough AEVs to have 10,001 spectral data in total. Using less exact eigenvalues leads to a bad recovery for \(q\) near the point \(x=1.\)  The execution time was around 2.3 seconds. The parameter \(K\) appearing in \eqref{eqn:bestfit} is also found to be 1 for this example as expected. \\
		
		Also, we consider the problem \(L(p,q,-\pi/2,\beta)\) assuming that the parameter \(\beta=\pi/4\) is not known beforehand. We considered 41 exact eigenvalues to be known. We obtained \(\vert \beta - \beta_{rec} \vert = 9.1\cdot 10^{-4}.\) As we can see from figure \ref{fig:Ex3beta} we obtain a similar accuracy to the previous scenario, adding enough AEVs to have 10,001 spectral data in total again. The computational time was 2.6 seconds.
		\begin{figure}
			\centering
			\includegraphics[width=5in,height=2in]{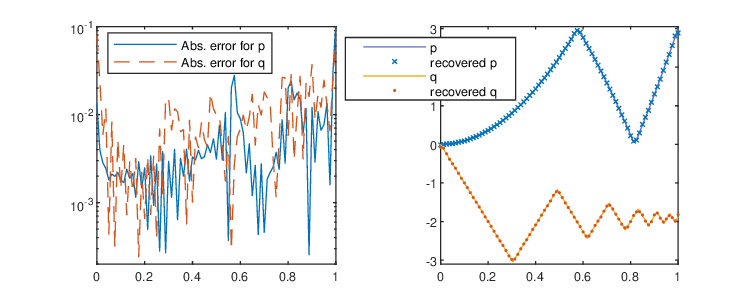}
			\caption{Absolute error of the recovery of the potential \eqref{eqn:example3} based on spline differentiation using 41 exact eigenvalues and complemented with 9960 AEVs to have 10,001 spectral data in total.  Right: Exact and recovered potential.}
			\label{fig:Ex3}
		\end{figure}
		\begin{figure}
			\centering
			\includegraphics[width=5in,height=2in]{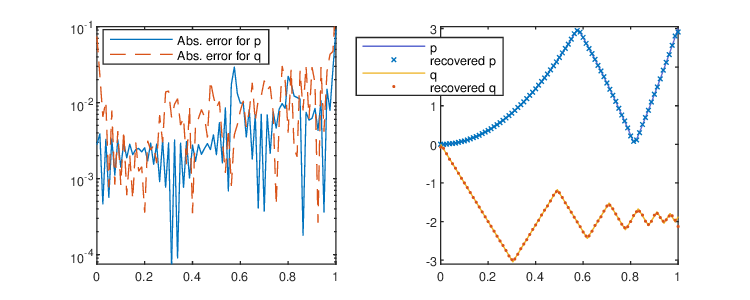}
			\caption{Left: Absolute error of the recovery of the potential \eqref{eqn:example3} based on spline differentiation using 41 exact eigenvalues and 9960 AEVs to have 10,001 spectral data in total. The parameter \(\beta=\pi/4\) is assumed to be unknown, \(\vert \beta - \beta_{rec} \vert = 9.1\cdot 10^{-4}.\)   Right: Exact and recovered potential. }
			\label{fig:Ex3beta}
		\end{figure}
	\end{Ex}
	
	\begin{Ex}\label{ex:sincos}
		Now let us consider a smooth oscillatory potential. For this example we assume \(\beta=0\) and known, that is, \(L(p,q,-\pi/2,0)\) with
		\begin{equation}\label{eqn:example4}
			p(x)=\sin 6x, \quad q(x)=\cos 8x.
		\end{equation}
		The aim of this example is to analyze what the parameter \(K\) appearing in \eqref{eqn:bestfit} is found to be by the algorithm for such a smooth potential. It turns out that given a sufficient amount of spectral data it works surprisingly well. Using 41 exact eigenvalues we obtain \(K=10.\) In figure \ref{fig:Ex4} we show the absolute error of the recovery based on spline differentiation using 41 exact eigenvalues and generating enough AEVs to have 10,001 spectral data in total. As we can see, the accuracy is remarkable for the majority of the interval, although it decreases significantly as we approach the point \(x=1\). The execution time is 2.6 seconds for this example.
		\begin{figure}
			\centering
			\includegraphics[width=5in,height=2in]{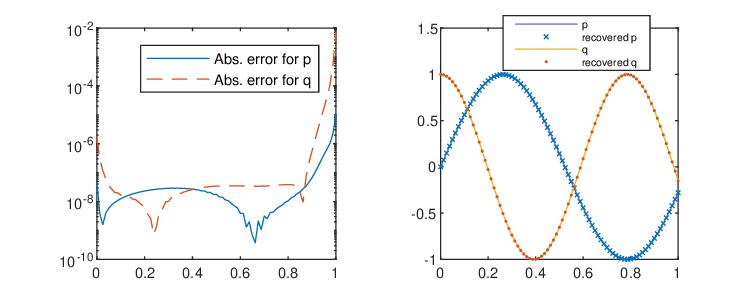}
			\caption{Left: Absolute error of the recovery of the potential \eqref{eqn:example4} based on spline differentiation using 41 exact eigenvalues and 9960 AEVs to have 10,001 spectral data in total.  Right: Exact and recovered potential.}
			\label{fig:Ex4}
		\end{figure}
	\end{Ex}
	\begin{Rmk}
		We observed numerically that whenever the coefficient \(a_1\) from \eqref{eqn:asym-normingc} is equal to 0, the performance of the proposed method greatly improves. One of the reasons for this is presented in remark \ref{rmk:nonconvergent}.  We conjecture that \(a_1 = -p(0) / \pi\) (we have strong numerical evidence for this claim, but we think the proof and detailed study of possible use of this formula are the subjects for a separate paper).
		In Example \ref{ex:sincos} the potential satisfies \(p(0) = 0\), and the recovered coefficient \(a_1\) is of order \(10^{-10}\), resulting in better numerical performance of the proposed method.
	\end{Rmk}
	%\pagebreak
	\appendix
	\section{The Bubnov-Galerkin process for stationary problems}\label{section:appendix-bg}
	Following \cite{mikhlin} we recall the Bubnov-Galerkin method for stationary problems. We consider the process for less general conditions than those appearing in \cite{mikhlin} for the sake of both clarity and brevity.  Let us consider an operator \( \mathcal{L}\) of the form \( \mathcal{L} = \mathcal{I} +\mathcal{K}\) defined on a separable Hilbert space \( \mathbb{H}\) whose inner product we denote by \( (\cdot, \cdot)\), \( \mathcal{I}\) denotes the identity operator and \(\mathcal{K}\) stands for a compact operator. Let us consider the equation
	\begin{equation}\label{eqn:bgmetheq}
		\mathcal{L}u-f = 0 \iff u+\mathcal{K}u -f = 0.
	\end{equation}
	Choose a coordinate system \(\{ \phi_k\}\) of \(\mathbb{H}\) such that
	\begin{enumerate}
		\item For arbitrary \(n\), the elements \( \phi_1, \phi_2, \ldots, \phi_n\) are linearly independent.
		\item \( \{ \phi_k\}\) is complete.
		\item \( \{\phi_k\} \) is strongly minimal. In particular, orthonormal systems are strongly minimal.
	\end{enumerate}
	An approximate solution of \eqref{eqn:bgmetheq} is sought in the form
	\begin{equation}\label{eqn:un-bg}
		u_n = \sum_{k=1}^{n}a_k^{(n)} \phi_k,
	\end{equation}
	and the coefficients \(a_k^{(n)} \) are such that after substitution of \(u_n\) for \(u\) in \eqref{eqn:bgmetheq}, the left-hand side of the equation is orthogonal to \( \phi_1, \phi_2, \ldots, \phi_n\). Then, the Bubnov-Galerkin system takes the form
	\begin{equation}
		\sum_{k=1}^{n} \left[ (\phi_k,\phi_j)+(K\phi_k,\phi_j) \right]a_k^{(n)} = (f,\phi_j), \quad j=1,\ldots,n.
	\end{equation}
	The following stability property of the Bubnov-Galerkin method is known.
	\begin{Thm}[{\cite[Thm. 14.1]{mikhlin}}]
		Let \(\{ \phi_k\}\) be a coordinate system that satisfies the three properties stated above. If equation \eqref{eqn:bgmetheq} has a unique solution, then the Bubnov-Galerkin method is stable.
	\end{Thm}
	
	\begin{Rmk}\label{rmk:stability}
		As we can see from theorem \ref{thm:glm}, the Gelfand-Levitan equation possesses a unique solution. Moreover, it follows from eq. \eqref{eqn:GL-IF} that the Gelfand-Levitan equation is of the form \eqref{eqn:bgmetheq}. Hence, application of the Bubnov-Galerkin method using a suitable coordinate system, such as the normalized Legendre polynomials, is stable.
	\end{Rmk}

	\bibliographystyle{plain}
	
\end{document}